\documentclass[a4paper,10pt]{amsart}
\usepackage{amsthm, amsmath,  amssymb, amsfonts, bbm,amscd}
\usepackage{color}
\usepackage{enumerate}
\usepackage[all]{xy}
\usepackage[latin1]{inputenc}
\usepackage{cite}
\usepackage{hyperref}
\usepackage{textcomp}
\usepackage[titletoc,toc]{appendix}

\numberwithin{equation}{section}

\newtheorem{theorem}{Theorem}[section]
\newtheorem{lemma}[theorem]{Lemma}
\newtheorem{definition-lemma}[theorem]{Definition-Lemma}
\newtheorem{proposition}[theorem]{Proposition}
\newtheorem{corollary}[theorem]{Corollary}
\newtheorem{example}[theorem]{Example}
\newtheorem{definition}[theorem]{Definition}
\newtheorem{remark}[theorem]{Remark}

\newcommand{\Hom}        {{\mathrm {Hom}}}

\newcommand{\Arrow}      {\mathrm{\to}}
\newcommand{\frakx}      {\mathfrak{X}}
\newcommand{\R}          {\mathbb{R}}

\newcommand{\G}          {\mathcal{G}}

\renewcommand{\H}        {\mathcal{H}} 
\newcommand{\sour}       {\mathsf{s}}
\newcommand{\V}          {\mathcal{V}}
\newcommand{\VB}         {\mathrm{VB}}
\newcommand{\tar}        {{\mathsf{t}}}

\newcommand{\m}          {\mathsf{m}}

\newcommand{\toto}       {\rightrightarrows}
\newcommand{\h}          {\mathbbm{h}}

\renewcommand{\v}        {\mathfrak{v}}
\renewcommand{\b}        {b}
\renewcommand{\c}          {\mathfrak{F}}
\newcommand{\F}          {\mathcal{F}}
\newcommand{\pr}         {\mathrm{pr}}
\newcommand{\<}          {\langle}
\renewcommand{\>}        {\rangle}
\newcommand{\Fl}         {\mathrm {Fl}}  
\newcommand{\frakh}      {\mathfrak{h}}
\newcommand{\bbD}        {\mathbb{D}}
\newcommand{\B}          {\mathcal{B}}
\newcommand{\E}          {\mathcal{E}}
\newcommand{\brac}[1]       {\<\!\< #1 \> \!\>}
\newcommand{\dbrac}[1]  {\| #1 \|}
\newcommand{\bG}        {\mathbb{G}}
\newcommand{\bM}        {\mathbb{M}}
\newcommand{\bA}        {\mathbb{A}}

\renewcommand{\uparrow} {\mathrm{v}}

\newcommand{\C}         {\mathcal{C}}

\newcommand{\Lie}        {\mathcal L}

\makeatletter
\newcommand{\doublewidetilde}[1]{{%
  \mathpalette\double@widetilde{#1}%
}}
\newcommand{\double@widetilde}[2]{%
  \sbox\z@{$\m@th#1\widetilde{#2}$}%
  \ht\z@=.9\ht\z@
  \widetilde{\box\z@}%
}
\makeatother

\begin{document}
\title{Differential forms with values in VB-groupoids}
\author{Thiago Drummond}
\email{drummond@im.ufrj.br}

\author[]{Leandro Egea}
%
\email{lgegea@icmc.usp.br}

\maketitle

\begin{abstract}
 We introduce multiplicative differential forms on Lie groupoids with values in VB-groupoids. Our main result gives a complete description of these objects in terms of infinitesimal data. By considering split VB-groupoids, we are able to present a Lie theory for differential forms on Lie groupoids with values in 2-term representations up to homotopy. We also define a differential complex whose 1-cocycles are exactly the multiplicative forms with values in VB-groupoids and study the Morita invariance of its cohomology.
\end{abstract}

\section{Introduction}
This paper is devoted to the study of differential forms on Lie groupoids with coefficients. While multiplicative differential forms with values in representations of Lie groupoids have been treated in the literature, see e.g. \cite{CSS}, here we consider the broader context of forms with values in $\VB$-groupoids. Recall that $\VB$-groupoids are, roughly, vector bundles in the category of Lie  groupoids. They naturally extend the notion of Lie-groupoid representations by enconding the information of representations {\em up to homotopy} \cite{Gra-Met1} (see also \cite{Arias-Crai1}), with the tangent bundle playing the role of the adjoint representation. The infinitesimal counterparts of $\VB$-groupoids are known as $\VB$-algebroids. In this paper, we introduce the notion of multiplicative differential form with values in $\VB$-groupoids and establish two main results: first, we provide a purely infinitesimal description of these objects, extending the works in \cite{AC, BC, CSS}; second, we describe a cohomology theory for such differential forms, that we prove to be Morita invariant.

 Multiplicative differential forms (with trivial coeffcients) on Lie groupoids appear
 in various contexts, often in connection with the (Lie-theoretic) integration of geometric structures:
 e.g. in symplectic and pre-symplectic groupoids \cite{Wein, bcwz}, which are the global objects integrating Poisson and Dirac structures, respectively. The Lie theory of multiplicative differential forms with trivial coefficients is by now totally understood \cite{AC, BC}. More recently, differential forms with coefficients in a representation and their infinitesimal versions, known as Spencer operators, were studied in \cite{CSS} in an effort to understand the work of Cartan on Lie pseudogroups from a global perspective (see also \cite{CO}). One important application of \cite{CSS} is providing the infinitesimal characterization of certain multiplicative distributions on Lie groupoids, a relevant topic due to its relation with quantization of Poisson manifolds \cite{Haw} and the Lie theory of Dirac structures \cite{Jotz-Ortiz, Jotz}. There is, however, an additional requirement the distributions studied in \cite{CSS} must satisfy: the tangent space of the manifold of units of the Lie groupoid must be contained in the distribution. By allowing more general $\VB$-groupoids as coefficients, we are able to drop this extra condition and develop a Lie theory for general multiplicative distributions. 

 There is yet another important feature of differential forms on Lie groupoids: their multiplicativity
 is a cocycle condition on a differential complex, known as the Bott-Shulman complex. Its cohomology is a Morita invariant of the Lie groupoid related
 to the de Rham cohomology of the associated classifying space \cite{Behrend}. Here, we introduce a differential complex for forms with coeffcients generalizing the Bott-Shulman complex and prove the Morita invariance of its cohomology. In view of the recent  work on the Morita invariance of $\VB$-groupoids \cite{Hoyo-Ort}, it is to be expected that this complex will be a useful tool in understanding connections on vector bundles over differentiable stacks.
 
\smallskip
 
\paragraph{\bf Statement of results.}
To explain our main results it is necessary to recall some facts regarding $\VB$-groupoids. These are presented in greater detail in Section 2.

Let $\G \toto M$ be a Lie groupoid and $\V \toto E$ be a $\VB$-groupoid over $\G$. We denote by $C \to M$ the \textit{core bundle of $\V$}; it is defined as the kernel of the source map (seen as a vector bundle morphism) $\widetilde{\sour}: \V|_M \to E$. The target map $\widetilde{\tar}: \V|_M \to E$ induces a vector bundle morphism $\partial: C \to E$ known as the \textit{core anchor}. 

The Lie algebroid of $\G$ and the $\VB$-algebroid of $V$ are denoted by $A \to M$ and $\v \to E$, respectively. Recall that $\v$ is also a double vector bundle, where the second vector bundle structure $\v \to A$ comes from applying the Lie functor to the structure maps of the vector bundle $\V \to \G$. The sections of $\v \to E$ which are vector bundle morphisms from $E \to M$ to $\v \to A$ are known as \textit{linear sections} and we denote them as $\Gamma_{lin}(E,\v)$. The projection of a linear section on the section of $A$ it covers is denoted by $\pr: \Gamma_{lin}(E,\v) \to \Gamma(A)$. 

A fundamental result regarding linear sections is that the right-invariant vector fields of $\V$ coming from $\Gamma_{lin}(E,\v)$ are linear \cite[Prop.~3.9]{Esp-Tor-Vit} (see also Proposition \ref{lemma:right_linear} below). So, there is a derivation $\Delta_\eta: \Gamma(\V) \to \Gamma(\V)$ and a Lie derivative operator $L_{\Delta_\eta}: \Omega(\G, \V) \to \Omega(\G, \V)$ corresponding to each element $\eta \in \Gamma_{lin}(E,\v)$. We review basic facts about linear vector fields and derivations on vector bundles in the Appendix. Recently, in \cite{Esp-Tor-Vit}, derivations of $\VB$-groupoids coming from multiplicative linear vector fields on $\V$ were studied. It is important to stress that their derivations are not related to ours.

We are now able to state our main result in the Lie theory of differential forms. First, an element $\vartheta \in \Omega^q(\G, \V)$ is said to be \textit{multiplicative} if it defines a $\VB$-groupoid morphism when seen as a map $T\G \oplus \dots \oplus T\G \to \V$. 

\begin{theorem}\label{thm:intro1}
If $\G \toto M$ is a source 1-connected groupoid, then there is a natural 1-1 correspondence between multiplicative forms $\vartheta \in \Omega^q(\G, \V)$ and triples $(D,l,\theta)$, where $l: A \to \wedge^{q-1} T^*M \otimes C$ is a vector bundle map, $\theta \in \Omega^q(M, E)$ and $D: \Gamma_{lin}(E,\v) \to \Omega^q(M, C)$ satisfies
\begin{equation}\label{comp_eq}
D(\B \Phi) = - \Phi \circ \theta, \,\, D(f \eta) = f D(\eta) + df \wedge l(\pr(\eta)), 
\end{equation}
where $f \in C^{\infty}(M), \,\, \eta \in \Gamma_{lin}(E,\v)$, $\B \Phi \in \Gamma_{lin}(E, \v)$ is the linear section covering the zero section of $A$ corresponding to a vector bundle morphism $\Phi: E \to C$. Also, the following equations hold, for $\eta_1, \, \eta_2 \in \Gamma_{lin}(E,\v)$, $\alpha, \, \beta \in \Gamma(A)$:
\begin{align*}
\tag{IM1} D([\eta_1, \eta_2]) & = L_{\nabla_{\eta_1}} D(\eta_2) - L_{\nabla_{\eta_2}} D(\eta_1)\\
\tag{IM2}l([\pr(\eta), \beta]) & = L_{\nabla_{\eta}} l(\beta) - i_{\rho(\beta)} D(\eta)\\
\tag{IM3}i_{\rho(\alpha)}l(\beta) & = - i_{\rho(\beta)} l(\alpha)\\
\tag{IM4}L_{\nabla_\eta} \theta & =  \partial(D(\eta))\\
\tag{IM5}i_{\rho(\alpha)} \theta & = \partial(l(\alpha)), 
\end{align*}
where $\nabla$ is the fat representation of $\Gamma_{lin}(E,\v)$ on $\partial: C \to E$, $\rho: A \to TM$ is the anchor of $A$. The triple $(D,l, \theta)$ is obtained from $\vartheta$ by the formulas:
$$
D(\eta)= L_{\Delta_\eta}(\vartheta)|_M, \,\, l(\alpha) = i_{\overrightarrow{\alpha}} \vartheta|_M, \,\, \theta = \vartheta|_M.
$$
\end{theorem}

We shall refer to the set of equations (IM1)-(IM5) as the \textit{IM equations} (where IM stands for ``infinitesimally multiplicative'') and to a triple $(D,l,\theta)$ satisfying them together with \eqref{comp_eq} as an \textit{IM $q$-form on $A$ with values in $\v$}. 

In section 3, we show how Theorem \ref{thm:intro1} recovers previous infinitesimal-global results in the literature regarding differential forms on Lie groupoids. In this section, we also show how multiplicative forms with values in $\VB$-groupoids and IM forms with values in $\VB$-algebroids give rise to a notion of multiplicative forms and IM forms with values in representations up to homotopy. Some of the equations appearing in this context have also appeared in \cite{Wald} in connection with higher gauge theory. 

A distribution $\H \subset T\G$ on the Lie groupoid $\G$ is called \textit{multiplicative} if it is a subgroupoid of the tangent groupoid $T\G$. These distributions can be studied via Theorem \ref{thm:intro1} by considering the quotient projection $\vartheta: T\G \to T\G/\H$ as a 1-form with value in a $\VB$-groupoid. In Section 6, we show how the resulting IM 1-form on $A$ with values in $\mathrm{Lie}(T\G/\H)$ can be refined to give an infinitesimal description of multiplicative distributions using the notion of \textit{IM distributions}. The infinitesimal-global correspondence between multiplicative and IM distributions (Theorem \ref{thm:IM_dist}) recovers the result of \cite{CSS} characterizing multiplicative distributions when the base manifold of $\H$ is $TM$. It is important to mention that the set of equations satisfied by IM distributions does not depend on the choice of connections on $A$. This improves the characterization of $\VB$-subalgebroids of $TA$ appering in \cite{Dru-Jotz-Ort}.


Our approach to prove Theorem \ref{thm:intro1} is built on ideas developed in \cite{BC, Bur-Drum}. Namely, the multiplicativity of $\vartheta \in \Omega^q(\G,\V)$ can be characterized by a cocycle equation for the corresponding function on the big groupoid 
$$
\bG = T\G \oplus \dots \oplus T\G \oplus \V^*
$$
and the corresponding IM $q$-form on $A$ with values in $\v$ can be obtained by taking Lie derivatives along right-invariant vector fields of $\bG$ coming from linear and other special sections of its Lie algebroid known as \textit{core} sections. This is done in Section 5.

The multiplicative forms with values in a $\VB$-groupoid are also 1-cocycles in a complex $C^{\bullet,q}(\V)$ defined as follows:
\begin{equation*}
\begin{aligned}
\C^{p,q}(\V) & = \{\vartheta \in \Omega^q(B_p \G, \pr_1^*\V) \,\,| \,\,\ \widetilde{\sour} \circ \vartheta = \partial_0^*\theta, \, \, \text{ for some } \theta \in \Omega^q(B_{p-1}\G, t^*E)\}, 
\end{aligned}
\end{equation*}
where $B_p\G$ is the space of composable $p$-arrows, $\pr_1: (g_1, \dots, g_p) \mapsto g_1$ is the projection on the first arrow, $t:(g_1,\dots, g_{p-1}) \mapsto \tar(g_1)$, for the target map, and $\partial_0: (g_1, \dots, g_p) \mapsto (g_2, \dots, g_p)$ is the 0th face map. Note that $C^{\bullet,0}(\V) = C_{\VB}^\bullet(\V)$, the $\VB$-groupoid complex introduced in \cite{Gra-Met1} to calculate the cohomology of $\G$ with values in 2-term representation up to homotopy. Also, the differential $\delta: C^{\bullet,q}(\V) \to C^{\bullet+1, q}(\V)$ (see \eqref{differential} for the explicit formula) is a natural extension of the differential on $C^{\bullet}_{\VB}(\V)$. Our second main result is the following:

\begin{theorem}\label{thm:intro2}
The cohomology of the complex $C^{\bullet,q}(\V)$ is a Morita invariant of $\G$.
\end{theorem}

The construction of the differential complex and some particular examples are presented in Section 4. The proof of Theorem \ref{thm:intro2} is contained in Section 5.4 and also explores the idea of seeing differential forms as functions in a bigger groupoid. In fact, we prove that $C^{\bullet, q}(\V)$ can be embedded as a subcomplex of the differentiable cochain complex of $\bG$ and use this fact to prove the result. As corollaries of Theorem \ref{thm:intro2} we recover the result of Morita invariance of $\VB$-groupoid cohomology obtained in \cite{Hoyo-Ort} and the Morita invariance of C\"ech cohomology  of $\G$ with values in the sheaf of $q$-differential forms $\Omega^q$ obtained in \cite{Behrend}.

\subsection*{Acknowledments}
This work is based on the second author's doctoral thesis \cite{Egea} carried out at IMPA under the supervision of Henrique Bursztyn and supported by a doctoral scholarship from CNPq. During the final stages of the paper, the second author held a post doc position at ICMC-USP supported by a fellowship from CAPES. 

The authors would like to thank Henrique Bursztyn for suggesting the problem as well as for many conversations which helped to improve this work.

\tableofcontents

\section{VB-groupoids and VB-algebroids}
In this preliminary section, we recall some properties of $\VB$-groupoids and $\VB$-algebroids focusing in the study of the right-invariant vector fields on $\VB$-groupoids coming from core and linear sections on the respective $\VB$-algebroids. It will be particularly important to understand how this relationship behaves under dualization. Our main reference here is \cite{Bur-Cab-Hoy} (see also \cite{Mckz2}).

\subsection{Definitions}
For a vector bundle $E \to M$, we shall refer to the multiplication by non-negative scalars $\h: \mathbb{R}_{\geq 0}\times E \Arrow E$, $\h_\lambda(e)=\lambda e$ as the \textit{homogeneous structure} on $E$. A \textit{$\VB$-groupoid} $\V \toto E$ over $\G \toto M$ is a vector bundle $p_\V: \V \to \G$ such that the homogeneous structure on $\V$ defines Lie groupoid morphisms for each $\lambda$.


We shall denote the source and target maps for the groupoid $\G \toto M$ as $\sour, \tar$ and, for $\V \toto E$, as $\widetilde{\sour}, \widetilde{\tar}$. As usual, we shall denote the multiplication on $\G$ by concatenation and on $\V$ by $\bullet$, the inversion using the $(\cdot)^{-1}$ notation and treat the unit maps as inclusions $M \subset \G$, $E \subset \V$. 

It is important to emphasize that the compatibility between the multiplication and the linear structure on $\V$ can be also presented as the \textit{interchange law}: if $(v_1, v_2), (w_1,w_2) \in \V_{g_1} \times \V_{g_2}$ are composable pairs of vectors, then $(v_1+w_1, v_2+w_2)$ is composable and 
\begin{equation}
(v_1 + w_1)\bullet (v_2 + w_2) = v_1 \bullet v_2 + w_1\bullet w_2.
\end{equation}

The \textit{core bundle} of $\V$ is the vector bundle $C \to M$ defined as $C:=\ker(\widetilde{\sour})|_M $. There is a short exact sequence of vector bundles over $\G$
\begin{equation}\label{core_ses}
0 \to \tar^*C \to \V \to \sour^*E \to 0
\end{equation}
called the \textit{core exact sequence}. The inclusion $\tar^*C \to \V$ defines a map $c \mapsto c_R$ on the level of sections given as
$
c_R(g) = c(\tar(g))\bullet 0_g.
$
The target map $\widetilde{\tar}$ restricted to $C$ defines a vector bundle map $\partial:= \widetilde{\tar}|_C: C \to E$ called the \textit{core anchor}. It will be important to introduce left-core sections $c_L \in \Gamma(\V)$ as well: they are defined by
$$
c_L(g) =  - 0_g \bullet c(\sour(g))^{-1} = 0_g \bullet (c- \partial(c))(\sour(g)), \,\,\, c \in \Gamma(C).
$$

\begin{example}\em
 For a Lie groupoid $\G \toto M$, its tangent bundle is an example of $\VB$-groupoid, $T\G \toto TM$. The core bundle is the Lie algebroid of $\G$, $A \to M$, with core anchor given by the anchor of the Lie algebroid, $\rho: A \to TM$. Also, for a section $\alpha \in \Gamma(A)$, $\alpha_R$ (resp. $\alpha_L) \in \mathfrak{X}(\G)$ is the right-invariant (resp. left-invariant) vector field which we denote here as $\overrightarrow{\alpha}$ (resp. $\overleftarrow{\alpha}$).
\end{example}

\begin{example}\label{semi-direct}\em
Let $\mathcal{E}=C[1]\oplus E$ \footnote{$C$ has degree $-1$ and $E$ has degree $0$.} be a graded vector bundle carrying a representation up to homotopy (ruth) of $\G \toto M$. We denote by $\partial: C\to E$ the 2-term complex, $\Psi_g$ the quasi-action on the complex,
$$
\begin{CD}
 C_{\sour(g)} @> \Psi_g >> C_{\tar(g)}\\
 @V \partial VV    @VV \partial V \\
 E_{\sour(g)} @> \Psi_g>>E_{\tar(g)},
\end{CD}
$$ 
and $\Omega_{g_1, g_2}: E_{\sour(g_2)} \to C_{\tar(g_1)}$ the curvature term associated to the representation (see \cite{Arias-Crai1, Gra-Met1}). \textit{The semidirect product of $\G$ with the ruth} is the vector bundle $\V = \tar^*C \oplus \sour^*E \to \G$ endowed with the $\VB$-groupoid structure $\V \toto E$ given by:
\begin{align}\label{Str}
\begin{split}
 \widetilde{\sour}(g, c, e) & = e, \,\, \widetilde{\tar}(g,  c, e) = g \cdot e  + \partial(c)\\
 (g_1, c_1, e_1) \bullet (g_2, c_2, e_2) & = (g_1g_2, c_1 + g_1 \cdot c_2 - \Omega_{(g_1,g_2)}(e_2), e_2)\\
 (g,c,e)^{-1} & = (g^{-1}, \,\,\Omega_{(g^{-1}, g)}(e) - g^{-1} \cdot c, \partial(c) + g \cdot e).
\end{split}
\end{align}
%
\end{example}

\begin{example}\em
A \textbf{double vector bundle} is a pair of vector bundles $\v \to E$ and $A \to M$ such that $\v$ is a $\VB$-groupoid over $A$ \footnote{A vector bundle can be seen as a Lie groupoid where $\sour=\tar=$ the bundle projection and multiplication is the fiberwise sum.}. In this case, one can prove that $\v \to A$ is also a $\VB$-groupoid over $E \to M$.
\end{example}

\paragraph{\bf Differentiation.}
A \textit{$\VB$-algebroid} is a pair of Lie algebroids $\v \to E$, $A \to M$ such that $p_\v: \v \to A$ is a vector bundle and the homogeneous structure on $\v$ defines Lie algebroid morphisms for each $\lambda$. 
In particular, a $\VB$-algebroid is a double vector bundle. Let $C_\v=\ker(\pi) \cap \ker(p_\v)$ be its core bundle and $\partial_\v: C_\v \to E$ be the restriction of the anchor map $\rho_\v: \v \to TE$ to the core bundle.

There is an inclusion $\B: \Gamma(C_\v) \hookrightarrow \Gamma(E,\v)$ defined as follows:
\begin{equation}\label{B_inj}
\mathcal{B}c (e) = 0_e +_A c(x), \,\, e \in E_x,
\end{equation}
where $+_A$ is the sum on the vector bundle $\v \to A$. We shall refer to such sections as \textit{core sections}.

The Lie functor applied to a $\VB$-groupoid gives rise to a $\VB$-algebroid \cite{Bur-Cab-Hoy}. In the next Proposition, we investigate how core sections integrate to $\VB$-groupoids.
\begin{proposition}\label{right_core}
If $\v = \mathrm{Lie}(\V)$, then 
$$
C_\v = C, \,\,\partial_\v = \partial.
$$ 
Moreover, for $c \in \Gamma(C)$, the right-invariant vector field $\overrightarrow{\mathcal{B}c} \in \mathfrak{X}(\V)$ is the vertical lift \footnote{On a vector bundle $E \to M$, any section $u \in \Gamma(E)$ gives rise to a vector field $u^\uparrow \in \frakx(E)$ called \textit{the vertical lift of $u$} whose flow is given by $e \mapsto e + \epsilon u(x)$, for $e \in E_x$.} of $c_R \in \Gamma(\V)$: 
$$
\overrightarrow{\mathcal{B}c} = c_R^\uparrow.
$$
In particular, $[\B c_1, \B c_2] = 0$, for $c_1, c_2 \in \Gamma(C)$.
\end{proposition}

\begin{proof}
The first statament follows from expressing $T\widetilde{\sour}, \, Tp_\V$ and $T\widetilde{\tar}$ under the decomposition 
$T_{0_x}\V = \V_x \oplus T_x \G = \V_x \oplus T_xM \oplus A_x$ and $T_{0_x}E = E_x \oplus T_xM$, for $x \in M, \, 0_x \in E_x$.

A section $c: M \to C$ induces a flow of bisections of $\V$, $b_{\epsilon}: E \to \V$, by
$$
b_{\epsilon}(e) = e+ \epsilon \,c(x), \,\,\, \text{ for } e \in E_x.
$$
Upon differentiation, one obtains $\B c (e)= \left.\frac{d}{d\epsilon}\right|_{\epsilon=0} b_{\epsilon}(e)$. So, the flow of $\overrightarrow{\mathcal{B}c}$ is $\Fl^\epsilon_{\overrightarrow{\B c}}(v) = b_\epsilon(\widetilde{\tar}(v)) \bullet v$. But, for $v \in \V_g$, using the interchange law,
\begin{align*}
b_\epsilon(\widetilde{\tar}(v)) \bullet v & = (\widetilde{\tar}(v)+ \epsilon \,c(\tar(g)))\bullet (v + 0_{g}) = v + \epsilon \, c(\tar(g))\bullet 0_g \\
& = v + \epsilon \, c_R(g).
\end{align*}
The result now follows from differentiation. At last, the statement regarding the Lie bracket of $\B c_1, \, \B c_2$ follows from the fact that $[\mu_1^\uparrow, \mu_2^\uparrow] = 0$, for any sections $\mu_1, \mu_2: \G \to \V$.
\end{proof}

In the following, we shall denote the inclusion $C_x \hookrightarrow \v_{0_x}$ by $c \mapsto \overline{c}$:
$$
\overline{c} = \left.\frac{d}{d\epsilon}\right|_{\epsilon=0}(\epsilon c) \in  \v \subset T_{0_x}\V, \,\,\, x \in M, 0_x \in E.
$$

\begin{remark}\em \label{Bphi}
One can extend the construction of core sections $\mathcal{B}c$ to any fiber preserving map $\Phi: E \to C$. Indeed, 
$
\mathcal{B}\Phi (e) = 0_e +_A \overline{\Phi}(e),
$
defines a section of $\v \to E$ coming from differentiating the flow of bisections $b_\epsilon: E \to \V$, $b_\epsilon(e) = e+ \epsilon \, \Phi(e)$. In the case $\Phi$ is a vector bundle morphism, it is straightforward to check that
$
\overrightarrow{\mathcal{B}\Phi} = W_{\Phi_R},
$
the linear vector field on $\V$ associated to the endomorphism $\Phi_R: \V \to \V$, $\Phi_R(v)= \Phi(\widetilde{\tar}(v))\bullet 0_g$, for $v \in \V_g$.
\end{remark}

\medskip

\subsection{Fat groupoid and fat algebroid.}
The theory of derivations and linear vector fields is recalled in the appendix for the convenience of the reader. 

\subsubsection{Definitions and the fat representations}
Following \cite{Gra-Met1}, we define the \textit{fat category} of $\V$ as the category whose space of objects is $\mathcal{F}(\V)$, the space of pointwise splittings of the core exact sequence \eqref{core_ses}. More precisely, an element of $\mathcal{F}(\V)$ is a pair $(g,\b)$, where
$
\b: E_{\sour(g)} \to \V_g 
$
is a linear map satisfying $\widetilde{\sour}\circ \b = \mathrm{id}$. There is a Lie category structure \footnote{A Lie category is defined exactly as a Lie groupoid, except for the existence of inverse.} on $\mathcal{F}(\V) \toto M$ with source, target maps induced by the source and target of $\G$ and multiplication given by:
$$
(g_1, \b_1)(g_2, \b_2) = (g_1 g_2, \b_1 \cdot \b_2), \,\, (\b_1 \cdot \b_2)(e) = \b_1(\widetilde{\tar}(\b_2(e)))\bullet \b_2(e).
$$ 
There is a natural Lie category representation of $\mathcal{F}(\V) \toto M$ on the vector bundle $\V|_M \to M$  given as follows: 
$$
\widehat{\Psi}_{(g,b)}(v) = b(\widetilde{\tar}(v)) \bullet v \bullet b(\widetilde{\sour}(v))^{-1}, \,\, v \in \V_{\sour(g)}
$$
The representation $\widehat{\Psi}$ preserves the decomposition $\V|_M = E \oplus C$, inducing a representation $\Psi$ on the core anchor complex $\partial: C \to E$ given by 
\begin{equation}\label{fat_rep2}
\Psi_{(g,\b)}(e) = \widetilde{\tar}(\b(e)), \;\; \Psi_{(g,b)}(c) = \b(\partial(c)) \bullet c \bullet 0_{g^{-1}}.
\end{equation}

The invertible elements $(g,b)$ of $\F(\V)$ are those which satisfy the additional condition: $\Psi_{(g,b)}: E_{\sour(g)} \to E_{\tar(g)}$ is an isomorphism; they define a Lie groupoid $\mathcal{F}_{\rm inv}(\V) \toto M$ called \textit{the fat groupoid}. The map $b^{-1}: E_{\tar(g)} \to \V_{g^{-1}}$ is given by
$$
b^{-1}(e) = b(\Psi_{b}^{-1}(e))^{-1}.
$$
Note that the bisections of $\mathcal{F}_{\rm inv}(\V) \toto M$ correspond bijectively to the bisections of $\V \toto E$ satisfying the additional property of being a vector bundle morphism from $E$ to $\V$. The projection $\mathcal{F}(\V) \to \G$ defines an affine bundle with fiber over $g \in \G$ modeled on ${\rm Hom}(E_{\sour(g)}, C_{\tar(g)})$. Also, \eqref{fat_rep2} induces a representation of the groupoid of invertible elements $\F_{\rm inv}(\V) \toto M$ on $\partial: C \to E$ called \textit{the fat representation}. 

\begin{example}\em
 For a representation of $\G \toto M$ on a vector bundle $C \to M$, consider the semidirect product $\tar^*C \toto M$ (see Example \ref{semi-direct}). In this case, $E=0$ and one can check that
 $$
 \F(\tar^*C) = \F_{\rm inv}(\tar^*C) \cong \G
 $$
 Also, the fat representation on $C$ is isomorphic to the representation  of $\G$ on $C$. 
\end{example}

Let us now study the infinitesimal picture. Let $\pi: \v \to E$ be a $\VB$-algebroid over $A \to M$ and recall that there is an underlying double vector bundle structure on $\v$. As such, one can define $\F(\v)$ and $\F_{\rm inv}(\v)$ as above. The fact that $\widetilde{\sour}= \widetilde{\tar} = \pi$ in this case implies that $\F(\v) = \F_{\rm inv}(\v)$. Also, $\F(\v)$ is a vector bundle which fits into an exact sequence
\begin{equation}\label{linear_ses}
0 \to \Hom(E,C) \to \F(\v) \stackrel{\pr} \to A \to 0,
\end{equation}
where the inclusion $\Hom(E,C) \to \F(\v)$ (at the level of sections) is exactly $\Phi \mapsto \mathcal{B}\Phi$ (see Remark \ref{Bphi}). In the following, we shall denote $\Gamma(\F(\v))$ as $\Gamma_{lin}(E,\v)$ and refer to these sections as \textit{the linear sections of $\v$}. It is important to recall that $\{\B c, \eta\}$, for $c \in \Gamma(C)$ and $\eta \in \Gamma_{lin}(E,\v)$, generate $\Gamma(E, \v)$ as a $C^\infty(E)$-module (see \cite[Prop.~2.2]{Mac-doubles}).

\begin{proposition}\label{lemma:right_linear}
If $\v = \mathrm{Lie}(\V)$, then $\mathcal{F}(\v)$ is the Lie algebroid of $\mathcal{F}_{\rm inv}(\V)$. Moreover, given a section $\eta \in \Gamma_{lin}(E,\v)$, the corresponding right-invariant vector field $\overrightarrow{\eta} \in \mathfrak{X}(\V)$ is a linear vector field.
\end{proposition}

\begin{proof}
A path on the source fiber of $\mathcal{F}_{\rm inv}(\V)$ starting at the identity section, $\b_{\epsilon}: E_x \to \V_{g_\epsilon}$, $\sour(g_\epsilon) = x$, defines upon differentiation a map $\mathfrak{b}: E_x \to \v_a$, where $a = \left.\frac{d}{d\epsilon}\right|_{\epsilon=0} g_\epsilon$. The linearity of $\mathfrak{b}$ follows from differentiating
$
\b_{\epsilon} \circ \h_{\lambda} = \h^\V_{\lambda} \circ \b_{\epsilon}
$
and the characterization of linear maps as those which commutes with the homogeneous structures \cite{Gra-Rot}. This proves that the Lie algebroid of $\F_{\rm inv}(\V)$ is contained in $\F(\v)$. Conversely, consider a linear map $\mathfrak{b}: E_x \to \v_a$ satisfying $\pi \circ \mathfrak{b} = {\rm id}$ and let $\eta: E \to \v$ be a linear section of $\v \to E$ such that $\eta|_{E_x} = \mathfrak{b}$. Define
\begin{equation}\label{bis_flow}
\b_\epsilon = \Fl^\epsilon_{\overrightarrow{\eta}}|_{E_x}: E_x \to \V_{\Fl_{\rho(\alpha)}^{\epsilon}},
\end{equation}
where $\alpha = \pr(\eta) \in \Gamma(A)$. As $\left.\frac{d}{d\epsilon}\right|_{\epsilon=0} \b_\epsilon  = \mathfrak{b}$, the result will follow if we prove that $\b_\epsilon$ is linear for every $\epsilon$.  It is straightforward to check that $\h^\V_{1/\lambda} \circ \Fl^\epsilon_{\overrightarrow{\eta}} \circ \h_{\lambda}$ is also the flow of a right invariant vector field, for all $\lambda > 0$. Also, 
$$
\left.\frac{d}{d\epsilon}\right|_{\epsilon=0} \h^\V_{1/\lambda} \circ \Fl^\epsilon_{\overrightarrow{\eta}} \circ \h_{\lambda}(e) =  \h^\v_{1/\lambda} \circ \eta \circ \h_{\lambda}(e) = \eta(e), \,\,\forall \, e \in E,
$$
by linearity of $\eta$. By uniqueness of integration, it follows that $\Fl^\epsilon_{\overrightarrow{\eta}} \circ \h_{\lambda} = \h^\V_\lambda \circ \Fl_{\overrightarrow{\eta}}^\epsilon$, which implies linearity both of $\b_\epsilon$ and $\overrightarrow{\eta}$. This concludes the proof. 
\end{proof}

We shall refer to $\F(\v)$ as \textit{the fat algebroid}. The next Proposition investigates how the fat representation \eqref{fat_rep2} differentiates to give a representation of $\F(\v)$ on $\partial: C \to E$.

\begin{proposition}\label{prop:fat_rep}
Let $\v = \mathrm{Lie}(\V)$ and $\nabla$ be the $\F(\v)$-connection on $\partial: C \to E$ obtained by differentiation of \eqref{fat_rep2}. Then,
\begin{equation}\label{fat_rep1}
\mathcal{B} (\nabla_\eta \,c) = [\eta, \mathcal{B}c],\, \,\,
\nabla_{\eta} \,u = \Delta_{\rho_v(\eta)}^\top(u), 
\end{equation}
where $\eta \in \Gamma_{lin}(E,\v),\, u \in \Gamma(E), \, c \in \Gamma(C)$ and $\Delta_{\rho_v(\eta)}^\top: \Gamma(E) \to \Gamma(E)$ is the adjoint derivation associated to the linear vector field $\rho_\v(\eta) \in \frakx(E)$. 
\end{proposition}

\begin{proof}
It is straightforward to check the equality restricted to $E$. Given $\eta \in \Gamma_{lin}(E,\v)$, the flow of the right-invariant (linear) vector field $\overrightarrow{\eta} \in \mathfrak{X}(\V)$, $\Fl_{\overrightarrow{\eta}}^\epsilon: \V \to \V$, preserve the subbundle $\tar^*C \hookrightarrow \V$. Indeed,
\begin{equation}\label{eq:pull_back}
\Fl_{\overrightarrow{\eta}}^\epsilon(c \bullet 0_g) = b_\epsilon(\partial c) \bullet c \bullet 0_{g} = (b_\epsilon(\partial c) \bullet c \bullet 0_{g_{\epsilon}^{-1}}) \bullet 0_{g_\epsilon g}
\end{equation}
where $b_\epsilon: E_{\tar(g)} \to \V_{g_\epsilon} \in \F_{\rm inv}(\V)$ is \eqref{bis_flow}. Note that we can rephrase \eqref{eq:pull_back} as follows: if $F: \tar^*C \to C$ is the pull-back map, then $F \circ \Fl_{\overrightarrow{\eta}}^\epsilon = \Psi_{(g_\epsilon, b_\epsilon)}\circ F$. This implies that  $\Delta^\top_{\overrightarrow{\eta}}(c_R) = (\nabla_\eta c)_R$, for $c \in \Gamma(C)$. But, from Proposition \ref{right_core} and \eqref{der_adjoint},
$$
\overrightarrow{\mathcal{B}(\nabla_\eta c)} = (\nabla_\eta c)_R^\uparrow = (\Delta^\top_{\overrightarrow{\eta}}(c_R))^\uparrow = [\overrightarrow{\eta}, c_R^\uparrow] = [\overrightarrow{\eta}, \overrightarrow{\mathcal{B}c}] = \overrightarrow{[\eta,\mathcal{B}c]}.
$$
This concludes the proof.
\end{proof}

\begin{remark}\em 
We would like to emphasize that it follows from Propositions \ref{right_core}, \ref{lemma:right_linear} and \ref{prop:fat_rep} that   
\begin{align}
\label{VB1} [\Gamma_{\rm lin}(E, \v), \Gamma_{\rm lin}(E, \v)] & \subset \Gamma_{\rm lin}(E, \v)\\
\label{VB2} [\Gamma_{\rm lin}(E, \v), \Gamma(C)] & \subset \Gamma(C)\\
\label{VB3} [\Gamma(C), \Gamma(C)] & = \{0\},
\end{align}
when $\v = \mathrm{Lie}(\V)$. These relations can be used as axioms to define a $\VB$-algebroid structure on a double vector bundle as was done in \cite{Gra-Met2}. From this point of view, it follows directly from the definition that formulas \eqref{fat_rep1} extend to arbitrary $\VB$-algebroids giving a representation on $\partial: C \to E$. It is important to point out that the equivalence between the definition of $\VB$-algebroids we use here and the one given in \cite{Gra-Met2} was proved in \cite{Bur-Cab-Hoy}.
\end{remark}

\subsubsection{Example: jet groupoid and jet algebroid.}

The jet groupoid $J^1\G$ is the Lie groupoid of 1-jets of bisections of $\G$. Its Lie algebroid is the jet algebroid $J^1A$, the space of 1-jets of sections of $A$. It is straightforward to see that
$$
J^1\G = \F_{\rm inv}(T\G), \,\,J^1A = \F(TA).
$$
The fat representation \eqref{fat_rep2} of $J^1\G$ on the anchor complex $\rho: A \to TM$ is called the adjoint representation \cite{CSS, Ev-Lu-We}. 

The jet prolongation $\Gamma(A) \ni \alpha \mapsto j^1 \alpha \in \Gamma(J^1A)$ splits the short exact sequence \eqref{linear_ses} on the  level of sections giving rise to the so called \textit{classical Spencer operator} $D^{\rm clas}: \Gamma(J^1A) \to \Omega^1(M,A)$. More concretely, for $\eta \in \Gamma(J^1A)$, $\alpha= \pr(\eta)$: 
\begin{equation}\label{Spencer_decomp}
\eta = j^1\alpha - \B D^{\rm clas}(\eta),
\end{equation}
where $\B D^{\rm clas}(\eta) \in \Gamma(J^1A)$ is the section coming from $D^{\rm clas}(\eta): TM \to A$ (see Remark \ref{Bphi}).

The fat representation of $J^1A$ on $\rho: A \to TM$ is given by:
\begin{equation}\label{jet_representation}
\nabla_{\eta} \,\beta = [\alpha, \beta] + D^{\rm clas}_{\rho(\beta)}(\eta), \,\,\, \nabla_{\eta}\, X = [\rho(\alpha), X] + \rho(D^{\rm clas}_X(\eta)), 
\end{equation}
for $ \alpha = \pr(\eta), \, \beta \in \Gamma(A),\, X \in \frakx(M)$.

From Lemma \ref{lemma:right_linear}, there is a linear vector field (hence a derivation) on $T\G$ associated to any section of $J^1A$. The following proposition gives a characterization of these derivations. 

\begin{proposition}
For $\eta \in \Gamma(J^1A)$, the corresponding derivation $\Delta_\eta: \Gamma(T\G) \to \Gamma(T\G)$ is given by
\begin{equation}\label{right_inv_der}
\Delta_\eta(U)(g) = [\overrightarrow{\pr(\eta)}, U](g) + D^{\rm clas}_{T\tar(U(g))}(\eta) \bullet 0_g,
\end{equation}
\end{proposition}

\begin{proof}
Let $\alpha: M \to A$. The linear section of $TA \to TM$ corresponding to $j^1\alpha$ is the differentiation of $\alpha$ as a map,  $T\alpha: TM \to TA$. As such, it is well-known that the right-invariant vector field on $T\G$ corresponding to $T\alpha$ is the tangent lift of $\overrightarrow{\alpha}$ \cite[Eq.~(45), \S 9.7]{Mckz2}. Hence, it follows from Example \ref{ex:tan_lift} that
\begin{equation}\label{jet_deriv}
\Delta_{j^1\alpha}^\top = [\overrightarrow{\alpha}, \cdot].
\end{equation}
Now, decompose a section $\eta \in \Gamma(J^1A)$ as $\eta = j^1\alpha - \B D^{\rm clas}(\eta)$ and the result will follow from Remark \ref{Bphi} applied to the morphism $D^{\rm clas}(\eta): TM \to A$; the change in sign is the result of taking adjoints (see Remark \ref{der_linear}). 
\end{proof}

\subsection{Dualization.}
For a $\VB$-groupoid $\V \toto E$, its dual has a $\VB$-groupoid structure $\V^* \toto C^*$. We shall prove here that the fat groupoids of $\V$ and $\V^*$ are isomorphic under an adjoint operation. The associated Lie algebroid morphism recovers the correspondence introduced by Mackenzie \cite{Mac-doubles} between linear sections of a double vector bundle and its dual. This will allow us describe the right-invariant vector fields of $\V^*$ coming from linear sections. 

Let us briefly recall the $\VB$-groupoid structure of $\V^* \toto C^*$. The source and target maps are determined by
$$
\<\widetilde{\sour}(\psi), c(\sour(g))\> = \<\psi, c_L(g)\>, \,\,\, \<\widetilde{\tar}(\psi), c(\tar(g))\> = \<\psi, c_R(g)\>, \,\, c \in \Gamma(C), \, \psi \in \V^*_g.
$$
The multiplication is determined by 
\begin{equation}\label{dual_mult}
\<\psi_1 \bullet \psi_2, v_1 \bullet v_2 \> = \<\psi_1, v_1\> + \<\psi_2, v_2\>, \,\,\, \psi_i \in \V_{g_i}^*, v_i \in \V_{g_i}, \, i=1,2.
\end{equation}
Using the decomposition $\V|_M = E\oplus C$, one has that the unit map $C^* \hookrightarrow \V^*|_M$ is the identification $C^* = \mathrm{Ann}(E)$. It is straightforward to check that $E^*$ can be identified with the core of $\V^* \toto C^*$ as follows: given $v \in \V_x, x \in M$, an element $\varphi \in E^*_x$ defines a covector in the core of $\V^*$ by
$$
\<\varphi, v\> = \<\varphi, \partial(c)+e\>, \text{ where } v = c+e \in C_x\oplus E_x.
$$ 
Note that the core anchor for $\V^* \toto C^*$ is $\partial^*:E^* \to C^*$.

\begin{example}\em
The dual $\VB$-groupoid corresponding to $T\G \toto TM$ is the cotangent Lie groupoid $T^*\G \toto A^*$. Also, the dual $\VB$-groupoid corresponding to the semi-direct product $\tar^*C \toto M$ is the (right) action groupoid $C^*\rtimes \G \toto C^*$.
\end{example}

\begin{example}\em
For a double vector bundle,
$$
\begin{CD}
\v @>>>A\\
@VVV   @VVV\\
E @>>> M
\end{CD}
$$
one can dualize both the horizontal and the vertical $\VB$-groupoid structure giving rise to two different $\VB$-groupoids which are also double vector bundles:
$$
\begin{CD}
\v_A^* @>>> A\\
@VVV @VVV\\
C^* @>>> M 
\end{CD}
\hspace{15pt}
\text{ and }
\hspace{15pt}
\begin{CD}
\v_E^* @>>> C^*\\
@VVV @VVV\\
E @>>> M.
\end{CD}
$$
There is a non-degenerate pairing $\dbrac{\cdot, \cdot}: \v_E^* \times_{C^*} \v_A^* \to \R$ given as
\begin{equation}\label{double_brac}
\dbrac{\gamma, \phi} = \<\gamma, d\>_E - \<\phi,d\>_A,\,\,\,
\end{equation}
where $d \in \v$ is any element such that $(\gamma,d,\phi) \in \v_E^*\times_E \v  \times_A \v_A^*$.
We refer to \cite[Chapter~9]{Mckz2} for details. 
\end{example}

The fat groupoids associated to $\V$ and $\V^*$ are isomorphic via an adjoint operation $\mathcal{A}: \F_{\rm inv}(\V) \to \F_{\rm inv}(\V^*)$. Given an invertible element $b: E_{\sour(g)} \to \V_g$, define $b^\top: C^*_{\sour(g)} \to \V^*_g$ as follows: for $v \in \V_g$,
\begin{equation}
 \<b^{\top}(\mu), v\> = \<\mu, \, \Psi_{b^{-1}}(c)\>, \,\,\, \text{ where } c = (v-b(\widetilde{\sour}(v)))\bullet 0_{g^{-1}} \in C_{\tar(g)}.
\end{equation}
Using the definitions of the structure maps $\widetilde{\sour}, \widetilde{\tar}: \V^* \to C^*$, one can prove that
$$
\widetilde{\sour}(b^\top(\mu)) = \mu, \,\,\, \widetilde{\tar}(b^\top(\mu)) =\Psi_{b^{-1}}^*(\mu). 
$$
Hence, $b^\top \in \F_{\rm inv}(\V^*)$ and we define $\mathcal{A}(b) = b^\top$. 

We list here some properties of $\mathcal{A}$, which the reader can check.
\begin{enumerate}
 \item $\<b^\top(\mu),b(e)\> =0$, for all $(e, \mu) \in E \times_M C^*$;
 \item $\Psi_{b^\top} = \Psi_{b^{-1}}^*$ acting on $\partial^*: E^* \to C^*$;
 \item $1_\V^\top = 1_{\V^*}$;
 \item $(b_1 \cdot b_2)^\top = b_1^\top \cdot b_2^\top$;
 \item $(b^\top)^\top = b$.
\end{enumerate}
In particular, $\mathcal{A}$ is a Lie groupoid morphism. We denote by $\mathfrak{a}=\mathrm{Lie}(\mathcal{A}): \F(\v) \to \F(\mathrm{Lie}(\V^*))$ its derivative, where $\mathrm{Lie}(\V^*) \to C^*$ is the $\VB$-algebroid of $\V^*$.

\begin{lemma}\label{lem:adj}
Let $\eta \in \Gamma_{lin}(E,\v)$. One has that 
$$
b_\epsilon = \Fl_{\overrightarrow{\eta}}^\epsilon|_{E_x}: E_x \to \V_{\Fl^\epsilon_{\overrightarrow{\alpha}}(x)}
\Leftrightarrow
b_\epsilon^\top = \Fl_{\overrightarrow{\eta}^\top}^\epsilon|_{C_x^*}: C_x^* \to \V^*_{\Fl^\epsilon_{\overrightarrow{\alpha}}(x)},
$$
where $\overrightarrow{\eta}^\top \in \mathfrak{X}(\V^*)$ is the adjoint linear vector field (see Appendix \ref{app:linear}). In particular, 
$$
\overrightarrow{\mathfrak{a}(\eta)} = \overrightarrow{\eta}^\top
$$
\end{lemma}

\begin{proof}
 First note that property (1) of $\mathcal{A}$ completely determines $b^\top$ (i.e. if $\<\xi, b(e)\>=0$, for every $e \in E_x$, then $\xi = b^\top(\mu)$, where $\mu = \widetilde{\sour}(\xi) \in C^*_x$). Now, using \eqref{eq:adj_flow}, for every $(e, \mu) \in E\times_M C^*$,
 $$
 \<\Fl_{\overrightarrow{\eta}^\top}^\epsilon(\mu), \Fl_{\overrightarrow{\eta}}^\epsilon(e)\> = \<\mu, e\> = 0. 
 $$
 This proves the result.
\end{proof}

We shall now give a detailed description of $\mathfrak{a}$. First of all, recall that there is an isomorphism of double vector bundles,
\begin{equation}\label{dual_identification}
i_\V : \mathrm{Lie}(\V^*) \stackrel{\sim}\longrightarrow \v^*_A,
\end{equation}
defined as follows (see \cite[Thm.~11.5.12]{Mckz2}):
$$
\<i_\V(\left.\frac{d}{d\epsilon}\right|_{\epsilon=0}\psi_\epsilon), \left.\frac{d}{d\epsilon}\right|_{\epsilon=0} v_\epsilon\>_A =  \left.\frac{d}{d\epsilon}\right|_{\epsilon=0} \<\psi_\epsilon, v_\epsilon\>, 
$$
for a path $(\psi_\epsilon, v_\epsilon) \in \V^* \times_\G \V$ on the source fiber starting at the identity section. From now on, we shall always make the identification $\v_A^* =\mathrm{Lie}(\V^*)$ without further notice.

It is well-known (see \cite{Mac-doubles}) that there exists a bijection between the linear sections of a double vector bundle $D$ and its dual $D_A^*$: 
\begin{equation}\label{adjoint_cor}
\begin{array}{rcl}
\Gamma_{lin}(E, D) & \stackrel{\sim}\longrightarrow & \Gamma_{lin}(C^*, D_A^*)\\
\eta & \longmapsto & \eta^\top.
\end{array}
\end{equation}
Let us briefly recall this correspondence. Given a linear section $\eta \in \Gamma_{lin}(E,D)$, the function $\ell_\eta \in C^{\infty}(D_E^*)$
$$
\ell_\eta(\gamma)= \<\gamma, \eta(e)\>_E, \,\,\, \text{ for } 
\begin{minipage}[h]{30pt}
\xy
{\ar@{|->}_{}(0,25)*++{\gamma};(10,25)*++{\mu}}\\
{\ar@{|->}_{}(0,25)*++{};(0,15)*++{e}}\\
{\ar@{|->}_{}(0,15)*++{};(10,15)*++{x}}\\
{\ar@{|->}_{}(10,25)*++{};(10,15)*++{}}\\
\endxy
\end{minipage}
\in D_E^*,
$$
is fiberwise linear not only with respect to the projection $D_E^* \to E$ but also with respect to $D_E^* \to C^*$. So, there exists a section $\eta^\top \in \Gamma_{lin}(C^*, D_A^*)$ such that $\ell_\eta(\varphi) = \dbrac{\varphi, \eta^\top(\mu)}$, for the bracket $\dbrac{\cdot, \cdot}$ \eqref{double_brac}. Now, one can see from the definition of $\dbrac{\cdot,\cdot}$ that $\eta^\top$ is completely determined by the condition
\begin{equation}\label{eq:dual}
 \<\eta^\top(\mu), \eta(e)\>_A =0, \,\,\, \forall \, (e,\mu) \in  E \times_M C^*.
\end{equation}

\begin{proposition}\label{right_adjoint}
 Let $\v=\mathrm{Lie}(\V)$, for a $\VB$-groupoid $\V \toto E$. One has that
 \begin{equation}
  \mathfrak{a}(\eta) = \eta^\top, \,\,\,\eta \in \Gamma_{lin}(E,\v).
  \end{equation}
\end{proposition}

\begin{proof}
Using Lemma \eqref{lem:adj}, it suffices to show that $\overrightarrow{\eta^\top}= \overrightarrow{\eta}^\top$. Let $\brac{\cdot, \cdot}: T(\V) \times_{T\G} T(\V^*) \to \R$ be the derivative of the natural bracket $\<\cdot, \cdot\>: \V \times_\G \V^* \to \R$. For any $(\psi, v) \in \V^*\times_\G V$, using the right-invariance of the flows and \eqref{dual_mult}, one has that
\begin{align*}
 \brac{ \overrightarrow{\eta^\top}(\psi), \overrightarrow{\eta}(v)} = & \left.\frac{d}{d\epsilon}\right|_{\epsilon=0} \< \Fl_{\eta^\top}^\epsilon(\widetilde{\tar}(\psi)) \bullet \psi, \Fl_{\eta}^\epsilon(\widetilde{\tar}(v)) \bullet v\> \\
  & = \left.\frac{d}{d\epsilon}\right|_{\epsilon=0} \left(\<\Fl_{\eta}^\epsilon(\widetilde{\tar}(v)), \Fl_{\eta^\top}^\epsilon(\widetilde{\tar}(\psi))\> + \< \psi, v\> \right)\\
  & = \<\eta^\top(\widetilde{\tar}(\psi)), \eta(\widetilde{\tar}(v))\>_A = 0.
\end{align*}
Now, from \eqref{eq:dual_linear}, it follows that $\overrightarrow{\eta^\top}$ must be the adjoint of $\overrightarrow{\eta}$ as we wanted to prove.
\end{proof}

One can now use $\eta^\top, \B \varphi$ to describe the Lie algebroid structure on $\v_A^* \to C^*$, in the case $\v=\mathrm{Lie}(\V)$. Note that the $\F(\v)$-connection on $\partial^*: E^* \to C^*$ is obtained by differentiation of $(\Psi_{b}^{-1})^*$. So, using that $\mathfrak{a}$ is a Lie algebroid morphism and Proposition \ref{prop:fat_rep}, one gets
\begin{align}\label{eq:dual_vb}
\begin{split}
 [\eta_1^\top, \eta_2^\top] & = [\eta_1,\eta_2]^\top, \,\,\, [\eta^\top, \B\varphi] = \B(\nabla_\eta^\top \varphi), \,\,\,[\B\varphi_1, \B \varphi_2]=0\\
 \rho_{\v^*_A}(\eta^\top) & = W_{\nabla_\eta}, \,\, \rho_{\v^*_A}(\B\varphi) = \partial^*(\varphi)^\uparrow
\end{split}
\end{align}
$\nabla$ is the $\F(\v)$-connection on $\partial: C \to E$ given by \eqref{fat_rep1} and $W_{\nabla_\eta} \in \frakx(C^*)$ is the linear vector field corresponding to $\nabla_\eta: \Gamma(C) \to \Gamma(C)$.

\begin{remark}\em
 For a general $\VB$-algebroid $\v$ (not necessarily integrable), the existence of a Lie algebroid structure on $\v_A^* \to C^*$ is the result of the double linearity of the Poisson bracket on $\v_E^* \to C^*$ (see \cite{Mac-doubles}[Eqs.~(19)]). Note that Mackenzie in \cite{Mac-doubles} chooses to work with $(\v_E^*)_{C^*}^*$ whereas we work directly with $\v_A^*$ using \eqref{double_brac}. This explains the difference in the sign between our expression for $\rho_{\v^*_A}(\B \varphi)$ and his.  
\end{remark}

\section{Multiplicative differential forms on Lie groupoids}
In this section, we define our main object of study and state our result regarding its Lie theory.

\subsection{Definition and examples}
\begin{definition}
Let $\V \toto E$ be a $\VB$-groupoid over $\G \toto M$. A differential $k$-form on $\G$ with values in $\V$, $\vartheta \in \Omega^k(\G, \V)$ is said to be multiplicative if the corresponding map
$$
\xy
{\ar@{->}_{}(0,12)*++{\bigoplus^k T\G}; (0,0)*++{\bigoplus^k TM}} \\
{\ar@{->}_{}(2,12)*++{\vphantom{\bigoplus^k T\G}}; (2,0)*++{\vphantom{\bigoplus^k TM}}}\\
{\ar@{->}_{}(5,12)*++{}; (15,12)*++{}}\\
{\ar@{->}_{}(5,0)*++{}; (15,0)*++{}}\\
{\ar@{->}_{}(16,12)*++{\!\!\!\V}; (16,0)*++{\!\!\!E}}\\
{\ar@{->}_{}(14,12)*++{\vphantom{V}}; (14,0)*++{\vphantom{E}}}
\endxy 
$$
is a Lie groupoid morphism. 
\end{definition}

\begin{remark}\em
It is straightforward to check that the base morphism $\oplus^k TM \to E$ is the map associated to a differential form $\theta \in \Omega^k(M,E)$ and 
\begin{equation}
\theta = \vartheta|_M.
\end{equation}
\end{remark}

\begin{example}\em
Let $C \to M$ be a representation of $\G \toto M$. A differential form $\vartheta \in \Omega^k(\G, \tar^*C)$ with values in the semi-direct product $\tar^*C \toto M$ is multiplicative if and only if
\begin{equation}\label{mult_Ezero}
(m^*\vartheta)_{(g_1,g_2)} = pr_1^*\vartheta + g_1 \cdot \, pr_2^*\vartheta,
\end{equation}
where $pr_1, pr_2: \G_{(2)} \to \G$ are the projections on the first and second components, respectively (in Lemma \ref{lemma:rep_vb} below we prove a generalization of this fact). The Lie theory of these multiplicative forms were studied in \cite{CSS}. Note that by considering $C= M \times \R$ with the trivial representation, one recovers the theory of multiplicative differential forms on $\G$ \cite{BC}.
\end{example}

There is a lift operation which takes multiplicative forms $\vartheta \in \Omega^k(\G, \V)$ with values in $\V$ to multiplicative forms on the fat groupoid $\F_{\rm inv}(\V)$ with values in the fat representation $C \to M$. Let us be more precise. Given $b: E_{\sour(g)} \to \V_g$, an element of $\F_{\rm inv}(\V)$, and $\xi_1, \dots, \xi_k \in T_{(g,b)}\F_{\rm inv}(\V)$, define $\widehat{\vartheta} \in \Omega^k(\F_{\rm inv}(\V), \tar^*C)$ by
\begin{equation}\label{lift_form}
\widehat{\vartheta}(\xi_1, \dots, \xi_k) \bullet 0_g = \vartheta(T\pr(\xi_1), \dots, T\pr(\xi_k)) - b(\theta(T\sour(\xi_1), \dots, T\sour(\xi_k))
\end{equation}
Let us show that $\widehat{\vartheta}$ is multiplicative. We shall restrict ourselves to the case $k=1$ for simplicity. The general case follows similarly. Given a composable pair of vectors $(\xi_1, \xi_2)$, where $\xi_i \in T_{(g_i, b_i)}\F_{\rm inv}(\V)$, $i=1,2$, one has that
\begin{align*}
\widehat{\vartheta}(\xi_1 \bullet \xi_2)\bullet 0_{g_1g_2} = \vartheta(U_1 \bullet U_2) - b_1(\widetilde{\tar}(b_2(\theta(X_2))))\bullet b_2(\theta(X_2)),
\end{align*}
where $T\pr(\xi_i)=U_i \in T_{g_i}\G$ and $T\sour(U_i)=X_i \in T_{\sour(g_i)}M$, $i=1,2$. Now, using the multiplicativity of $\vartheta$ and the interchange law, it follows that
\begin{align*}
\widehat{\vartheta}(\xi_1 \bullet \xi_2)\bullet 0_{g_1g_2} & = (\vartheta(U_1) - b_1(\widetilde{\tar}(b_2(\theta(X_2)))))\bullet (\vartheta(U_2)-b_2(\theta(X_2)))\\
 &  = (\vartheta(U_1) - b_1(\theta(X_1))) \bullet 0_{g_2}  + b_1(\partial(\widehat{\vartheta}(\xi_2)))\bullet \widehat{\vartheta}(\xi_2)\bullet 0_{g_1^{-1}}0_{g_1g_2}\\
 & = (\widehat{\vartheta}(\xi_1)  +  b_1\cdot \widehat{\vartheta}(\xi_2)) \bullet 0_{g_1g_2}. 
\end{align*}
We have used in the second equality that $\partial(\widehat{\vartheta}(\xi_2)) = \theta(X_1) - \widetilde{\tar}(b_2(\theta(X_2)))$.
\begin{example}\em
 Let $\V = T\G \toto TM$ be the tangent groupoid. Differential forms $\vartheta \in \Omega^k(\G, T\G)$ with values in $T\G$ are also known as vector valued forms. The Lie theory of multiplicative vector valued forms were studied in \cite{Bur-Drum} (see also \cite{Bur-Drum2}). The lift operation $\vartheta \mapsto \widehat{\vartheta}$ takes multiplicative vector valued forms to multiplicative forms on the jet groupoid $J^1\G$ with values in the adjoint representation $\tar^*A$. The lift of $\mathrm{id} \in \Omega^1(\G, T\G)$ is exactly the Cartan 1-form on the jet groupoid $J^1\G$.
\end{example}

\subsection{Infinitesimal multiplicative forms on Lie algebroids}
Let $(A,\rho, [\cdot, \cdot])$ be a Lie algebroid and $\v \to E$ be a $\VB$-algebroid over $A$

\begin{definition}
 An IM $k$-form on $A$ with values on $\v$ is a triple $(D,l,\theta)$ where $D: \Gamma_{lin}(E,\v) \to \Omega^k(M, C)$, $l: A \to \wedge^{k-1} T^*M \otimes C$ and $\theta \in \Omega^k(M, E)$ satisfying
\begin{equation}\label{eq:compatibility}
D(\mathcal{B}\Phi) = -\Phi \circ \theta, \,\,\,\, D(f\eta) = fD(\eta) + df \wedge l(\pr(\eta))
\end{equation}
and the set of equations known as {\bf IM-equations}
\begin{align}
\tag{IM1} D([\eta_1, \eta_2]) & = L_{\nabla_{\eta_1}} D(\eta_2) - L_{\nabla_{\eta_2}} D(\eta_1)\\
\tag{IM2} l([\pr(\eta), \beta]) & = L_{\nabla_{\eta}} l(\beta) - i_{\rho(\beta)} D(\eta)\\
\tag{IM3} i_{\rho(\alpha)}l(\beta) & = - i_{\rho(\beta)} l(\alpha)\\
\tag{IM4} L_{\nabla_\eta} \theta & =  \partial(D(\eta)) \text{ (redundancies?) } \\
\tag{IM5} i_{\rho(\alpha)} \theta & = \partial(l(\alpha)), 
\end{align}
where $\nabla$ is the fat representation of $\F(\v)$ on $\partial:C \to E$ and $L_{\nabla_\eta}$ is the operator \eqref{der_Lie}.
\end{definition}

Let us give some examples.
\begin{example}\em[Spencer operators]\label{ex:spencer}
Let $\v = A \times_M C \to M$ be a $\VB$-algebroid with $E=0$. In this case, the projection $\pr: \Gamma_{lin}(E,\v) \ni \eta \mapsto \alpha \in \Gamma(A)$ gives an isomorphism of Lie algebroids $\F(\v) \cong \Gamma(A)$. The fat representation gives a flat $A$-connection $\nabla: \Gamma(A) \times \Gamma(C) \to \Gamma(C)$.

An IM $k$-form with values in $A \times_M C$ is a pair $(D,l)$ such that $D: \Gamma(A) \to \Omega^k(M,C)$, $l: A \to \wedge^{k-1} T^*M \otimes C$ satisfying the compatibility $D(f\alpha) = fD(\alpha) + df \wedge l(\alpha)$ and the equations 
\begin{align*}
 D([\alpha_1, \alpha_2]) & = L_{\nabla_{\alpha_1}} D(\alpha_2) - L_{\nabla_{\alpha_2}} D(\alpha_1)\\
 l([\alpha_1, \alpha_2]) & = L_{\nabla_{\alpha_1}}l(\alpha_2) - i_{\rho(\alpha_2)} D(\alpha_1)\\
 i_{\rho(\alpha_1)}l(\alpha_2) & = - i_{\rho(\alpha_2)}l(\alpha_1)
\end{align*}
In \cite{CSS}, such pairs are called \textit{$C$-valued Spencer operators on $A$.}
\end{example}

\begin{remark}\em
 Note that (IM1), (IM2), (IM3) in the definition of an IM $k$-form on $A$ with values in a general $\VB$-algebroid $\v$ can be interpreted as saying that $(D,l\circ \pr)$ is a $C$-valued Spencer operator on the fat Lie algebroid $\F(\v)$. 
\end{remark}

\begin{example}\em \label{ex:vector_val}
An \textit{IM $(1,k)$-form} on $A$ is a triple $(\mathfrak{D}, \mathfrak{l}, \mathfrak{r})$, where $\mathfrak{D}: \Gamma(A) \to \Omega^k(M, A)$, $\mathfrak{l}: A \to \wedge^{k-1} T^*M \otimes A$, $\mathfrak{r}: T^*M \to \wedge^k T^*M$
satisfy the Leibniz condition
$$
\mathfrak{D}(f\,\alpha) = f\, \mathfrak{D}(\alpha) + df \wedge \mathfrak{l}(\alpha) - \alpha \wedge \mathfrak{r}(df)
$$
and \textit{the IM-equations}:

\begin{minipage}[t]{5cm}
\begin{align*}
\mathfrak{D}([\alpha, \beta]) & = \alpha \cdot \mathfrak{D}(\beta) - \beta \cdot \,\mathfrak{D}(\alpha);\\
\mathfrak{l}([\alpha,\beta]) & = \alpha \cdot \mathfrak{l}(\beta) - i_{\rho(\beta)} \mathfrak{D}(\alpha);\\
\mathfrak{r}(\Lie_{\rho(\alpha)} \omega) & = \alpha \cdot \mathfrak{r}(\omega) - \<\omega, \rho(\mathfrak{D}(\alpha))\>; \\
\end{align*}
\end{minipage}
\begin{minipage}[t]{5cm}
\begin{align*}
i_{\rho(\alpha)} \mathfrak{l}(\beta) & = - i_{\rho(\beta)} \mathfrak{l}(\alpha);\\
i_{\rho(\alpha)} \mathfrak{r}(\omega) & =  \langle \omega, \rho(\mathfrak{l}(\alpha) \rangle.
\end{align*}
\end{minipage}

Here, $\,\cdot\,$ is the action of the Lie algebra $\Gamma(A)$ on $\Omega^\bullet(M, A)$ given by
$
\alpha \cdot (\omega \otimes \beta) = \Lie_{\rho(\alpha)}\omega \otimes \beta + \omega \otimes [\alpha, \beta].
$
These objects were introduced in \cite{Bur-Drum} as the infinitesimal data associated to multiplicative vector valued forms on Lie groupoids. 

There is a 1-1 correspondence between IM $(1,k)$-forms on $A$ and IM $k$-forms on $A$ with values in $TA$. Indeed, for a triple  $(\mathfrak{D}, \mathfrak{l}, \mathfrak{r})$, define

$$
l = \mathfrak{l}, \,\, \theta =  \mathfrak{r}^*, \,\, D(\eta) = \mathfrak{D}(\alpha) + D^{\rm clas}(\eta) \circ \theta, 
$$
where $\alpha = \pr(\eta)$ and $D^{\rm clas}$ is the classical Spencer operator. One can check that $(\mathfrak{D}, \mathfrak{l}, \mathfrak{r})$ is an IM $(1,k)$-form on $A$ if and only if $(D,l, \theta)$ is an IM $k$-form on $A$ with values on $TA$. Note that $\mathfrak{D}(\alpha) = D(j^1\alpha)$.
\end{example}

\subsection{Infinitesimal-global correspondence}
We are now able to state our result regarding the correspondence between multiplicative forms on Lie groupoids and IM forms on Lie algebroids.

\begin{theorem}\label{thm:main}
For source 1-connected Lie groupoid $\G \toto M$, there is a 1-1 correspondence between multiplicative $\vartheta \in \Omega^k(\G, \V)$ and IM $k$-forms $(D, l, \theta)$ on $A$ with values on $\v=\mathrm{Lie}(\V)$. The correspondence is given by
\begin{equation}\label{eq:inf_comp}
D(\eta)  = L_{\Delta_\eta}(\vartheta)|_M, \,\,\, l(\alpha)  = i_{\overrightarrow{\alpha}} \vartheta|_M, \,\,\, \theta  = \vartheta|_M,
\end{equation}
where $\Delta_{\eta}: \Gamma(\V) \to \Gamma(\V)$ is the (adjoint) derivation corresponding to the linear vector field $\overrightarrow{\eta} \in \frakx(\V)$ and $L_{\Delta_\eta}$ is the operator \eqref{der_Lie}.
\end{theorem}

We shall postpone the proof of Theorem \ref{thm:main} to \S \ref{proof1} and will focus here on how it recovers some particular results when one restricts the $\VB$-groupoid to special cases. It is specially important to obtain explicit expressions for the derivation $\Delta_\eta$ on these cases. 

\begin{corollary}
 Let $\G \toto M$ be a source 1-connected groupoid with a representation on $C \to M$. There is a 1-1 correspondence between differential forms $\vartheta \in \Omega^k(\G, \tar^*C)$ satisfying equation \eqref{mult_Ezero} and $C$-valued Spencer operators $(D, l)$ on $A$. The correspondence is given by
 \begin{align*}
 D(\alpha)(X_1,\dots,X_k) & = \left.\frac{d}{d\epsilon}\right|_{\epsilon=0} (\Fl_{\overrightarrow{\alpha}}^\epsilon(x))^{-1}\cdot \vartheta(T\Fl^\epsilon_{\overrightarrow{\alpha}}(X_1), \dots, T\Fl^\epsilon_{\overrightarrow{\alpha}}(X_k))\\
 l(\alpha) & = i_{\overrightarrow{\alpha}} \vartheta|_M.
 \end{align*}
\end{corollary}
\begin{proof}
Consider the semi-direct product $\V = \tar^*C \toto M$. One has that $A \times_M C = \mathrm{Lie}(\tar^*C)$ (see Example \ref{ex:spencer}). In this case, the flow of the linear vector field $\overrightarrow{\eta}$ corresponding to a linear section $\eta \cong \alpha \in \Gamma(A)$ is given by:
$$
\Fl_{\overrightarrow{\eta}}^\epsilon(g,c) = (\Fl_{\overrightarrow{\alpha}}^\epsilon(g), \Fl_{\overrightarrow{\alpha}}^\epsilon(\tar(g))\cdot c).
$$
Hence, for $X_1, \dots, X_k \in T_xM \subset T_x\G$, one has that
\begin{align*}
 (L_{\Delta_\eta}\vartheta)(X_1,\dots, X_k) & = \left.\frac{d}{d\epsilon}\right|_{\epsilon=0} \Fl_{\overrightarrow{\eta}}^{-\epsilon}(\vartheta(T\Fl^\epsilon_{\overrightarrow{\alpha}}(X_1), \dots, T\Fl^\epsilon_{\overrightarrow{\alpha}}(X_k)))\\
 & = \left.\frac{d}{d\epsilon}\right|_{\epsilon=0}\Fl_{\overrightarrow{\alpha}}^{-\epsilon}(\Fl_{\rho(\alpha)}^\epsilon(x)) \cdot \vartheta(T\Fl^\epsilon_{\overrightarrow{\alpha}}(X_1), \dots, T\Fl^\epsilon_{\overrightarrow{\alpha}}(X_k))\\
 & = \left.\frac{d}{d\epsilon}\right|_{\epsilon=0} (\Fl_{\overrightarrow{\alpha}}^\epsilon(x))^{-1}\cdot \vartheta(T\Fl^\epsilon_{\overrightarrow{\alpha}}(X_1), \dots, T\Fl^\epsilon_{\overrightarrow{\alpha}}(X_k)).
\end{align*}
Note that we are identifying the vertical subspaces of both $T_{(x,c)}(\tar^*C)$ and $T_c C$ with $C_x$. 
\end{proof}

\begin{example}\em
 The IM $1$-form on $A$ with values in $TA$ corresponding to $\mathrm{id}_\G \in \Omega^1(\G, T\G)$ is exactly $(D^{\rm clas}, \mathrm{id}_A, \mathrm{id}_{TM})$. This follows directly from \eqref{right_inv_der}. This correspondence should be seen as an alternative way to express the fact established in \cite{CSS} that $(D^{\rm clas}, \pr)$ is the Spencer operator on $J^1A$ with values in $A$ associated to the Cartan 1-form on the jet groupoid.
\end{example}

\begin{corollary}\label{cor:vec_val}
On a source 1-connected groupoid $\G \toto M$, there is a 1-1 correspondence between multiplicative vector valued forms $\vartheta \in \Omega^k(\G, T\G)$ and IM $(1,k)$-forms $(\mathfrak{D}, \mathfrak{l}, \mathfrak{r})$ given by
$$
\mathfrak{D}(\alpha) = [\overrightarrow{\alpha}, \vartheta]|_M, \,\,\mathfrak{l} = i_{\overrightarrow{\alpha}}\vartheta|_M, \,\,\, \mathfrak{r}(\mu) = \langle\tar^*\mu, \vartheta\rangle|_M,
$$
$[\cdot, \cdot]$ is the Fr\"olicher-Nijenhuis bracket on $\Omega^k(\G, T\G)$.
\end{corollary}

\begin{proof}
The 1-1 correspondence follows directly from Theorem \ref{thm:main} and Example \ref{ex:vector_val}. One only has to prove the formula for $\mathfrak{D}$. Now, if $D: \Gamma(J^1A) \to \Omega^k(M, A)$ is the first component of the IM $k$-form on $A$ with values in $TA$ corresponding to $\vartheta$, then
$$
\mathfrak{D}(\alpha) = D(j^1 \alpha) = L_{\Delta_{\overrightarrow{j^1\alpha}}} \vartheta|_M.
$$
Using that $\Delta_{\overrightarrow{j^1\alpha}} = \Lie_{\overrightarrow{\alpha}}$(see \eqref{jet_deriv}, one can check that the operator $L_{\Delta_{\overrightarrow{j^1\alpha}}}$ on $\Omega^\bullet(\G, T\G)$ is exactly $[\overrightarrow{\alpha}, \cdot]$.
\end{proof}


\subsection{Coefficients in a representation up to homotopy}
We here show how multiplicative forms with values in semi-direct products of groupoids with ruth give rise naturally to a notion of multiplicative forms with values in ruth.

\begin{definition}
A multiplicative $k$-form on $\G$ with values in a ruth $\mathcal{E}=C[1]\oplus E$ is a pair $\omega \in \Omega^k(\G, \tar^*C)$, $\theta \in \Omega^k(M, E)$ satisfying
\begin{align}
\partial\circ\omega_g  & =  \tar^*\theta - \Psi_g \circ \sour^*\theta \label{mul_forms1}\\ 
\m^*\omega_{(g_1,g_2)} & = pr_1^* \omega + \Psi_{g_1} \circ pr_2^*\omega - \Omega_{g_1,g_2}\circ s^* \theta, \label{mul_forms2}
\end{align}
where $s: \G_{(2)} \to M$ is the map $s(g_1,g_2)=\sour(g_2)$ and $(\partial, \Psi, \Omega)$ are the structure operators of the ruth.
\end{definition}

Let us give some examples.

\begin{example}\em
Any $\lambda \in \Omega^k(M,C)$ defines a multiplicative form on $\G$ with values in $\E$ as follows:
$$
\theta = \partial \circ \lambda, \,\,\, \omega_g = \tar^*\lambda - \Psi_g  \circ \sour^*\lambda, \,\,g \in \G.
$$
\end{example}

\begin{example}\em
Let $\Psi$ be a representation of $\G$ on the complex $\partial: C \to E$. It can be considered as a ruth on $\mathcal{E}=C[1]\oplus E$ with $\Omega=0$. In this case, a multiplicative $k$-form with values in $\mathcal{E}$ is a pair $\omega \in \Omega^k(\G, \tar^*C)$, $\theta \in \Omega^k(M, E)$ satisfying:
\begin{align}\label{eq:omega_zero}
\begin{split}
\partial\circ\omega_g & = \tar^*\theta - \Psi_g \circ \sour^*\theta\\ 
\m^*\omega_{(g_1,g_2)} & = pr_1^* \omega + \Psi_{g_1} \circ pr_2^*\omega.
\end{split}
\end{align}
Note that the case $E=0$ recovers the multiplicative differential forms with values in representations studied in \cite{CSS}. Also, when $E=0$ and $C=M \times \mathbb{R}$ with the trivial representation, one recovers the multiplicative forms studied in \cite{AC, BC}.
\end{example}

Regarding the last example, we shall point out that equations \eqref{eq:omega_zero} appeared in \cite{Wald} in the context of higher gauge theory, where they were used to define forms on Lie groupoids with values in Lie 2-algebras.

\begin{example}\em
Consider the differentiable cohomologies (see \cite{AC} and references therein) $H^p(\G_{(\bullet)})$ and $H^p(\Omega^k(\G_{(\bullet)})$ associated to $\G$ . An element $\psi \in H^2(\G_{(\bullet)})$ defines a ruth on the complex $0:C\to E$, where $C=E=M\times \mathbb{R}$, the quasi-action is the trivial representation and 
$$
\Omega_{g_1,g_2} = \psi(g_1,g_2) \in \mathbb{R} \cong \Hom(E_{\sour(g_2)}, C_{\tar(g_1)}).
$$
The cup product with $\psi$ defines a map $\ast_{\psi}: H^0(\Omega^k(\G_{(\bullet)}) \to H^2(\Omega^k(\G_{(\bullet)})$ given by
$$
(\ast_\psi \, \theta)_{(g_1,g_2)} = \psi(g_1,g_2) \,s^*\theta.
$$
We claim that multiplicative k-forms with values in the ruth determined by $\psi$ correspond to elements of $\ker(\ast_{\psi})$. Indeed, $\theta \in \Omega^k(M)$ is a cocycle if and only if it satisfies $\tar^*\theta - \sour^*\theta =0$ and the fact that $[\ast_\psi\,\theta]=0$ in cohomology is equivalent to the existence of $\omega \in \Omega^k(\G)$ such that $\pr_1^*\omega - m^*\omega|_{(g_1,g_2)} + \pr_2^*\omega = \psi(g_1,g_2) \,s^*\theta$.
\end{example}

\begin{lemma}\label{lemma:rep_vb}
 Let $\mathcal{E}= C[1]\oplus E$ be a ruth of $\G$ and consider the associated $\VB$-groupoid $\V = \tar^*C  \oplus \sour^*E \toto E$. There is a 1-1 correspondence between multiplicative $k$-forms $\theta \in \Omega^k(M, E), \omega \in \Omega^k(\G, \tar^*C)$ with values in $\mathcal{E}$ and multiplicative $k$-forms $\vartheta \in \Omega^k(\G, \tar^*C  \oplus \sour^*E)$ with values in $\V$. The correspondence is given by
 $$
 \vartheta(U_1, \dots, U_k) = (\omega(U_1, \dots, U_k), \theta(T\sour(U_1), \dots, T\sour(U_k))).
 $$
\end{lemma}

\begin{proof}
 It is a straightforward consequence of the formulas \eqref{Str}. More precisely, equation \eqref{mul_forms1} is equivalent to the compatibility of $\vartheta$ with the target map and \eqref{mul_forms2} is equivalent to the compatibility of $\vartheta$ with the multiplication.
\end{proof}

Let us now recall briefly how any $\VB$-groupoid can be presented as a semi-direct product. Let $\V \toto E$ be a $\VB$-groupoid over $\G \toto M$. A horizontal lift is a splitting $h: \sour^*E \to \V$ of the short exact  sequence \eqref{core_ses} such that $h|_M: E \hookrightarrow \V$ is the unit. Equivalently, $h$ can be seen as a section of the projection $\F(\V) \to \G$ preserving source, target and unit maps. In general, $h: \G \to \F(\V)$ will not preserve multiplication. For a composable pair $(g_1, g_2) \in \G_{(2)}$, there is a curvature term $\Omega_{(g_1,g_2)}: E_{\sour(g_2)} \to C_{\tar(g_1)}$ such that
$$
h(g_1 g_2) - h(g_1) \cdot h(g_2) = \Omega_{(g_1,g_2)} \bullet 0_{g_1g_2}
$$

A horizontal lift also induces a quasi-action $\Psi$ of $\G$ on the core complex $\partial: C \to E$: 
$$
\begin{CD}
C_{\sour(g)} @>  \Psi: \,c \,\mapsto\, h(g,\partial(c)) \bullet c \bullet 0_{g^{-1}} >> C_{\tar(g)}\\
@V \partial VV              @VV \partial V \\
E_{\sour(g)}  @>> \Psi: \,e \, \mapsto\, \widetilde{\tar}(h(g,e)) > E_{\tar(g)}.
 \end{CD}
$$
The quasi-action $\Psi$, the curvature $\Omega$ and the core-anchor $\partial: C \to E$ define a ruth of $\G$ on the $\E=C[1]\oplus E$. Moreover, the map
$$
\tar^*C \oplus \sour^*E \to \V, \,\,\,  (g, c, e) \mapsto c\bullet 0_g + h(g,e)
$$
is an isomorphism of $\VB$-groupoids.
This construction establishes an equivalence of categories between 2-term ruth and $\VB$-groupoids \cite{Hoyo-Ort,Gra-Met1}.

%
%


Let us move to the infinitesimal picture. Let $\mathcal{E}=C[1]\oplus E$ be a ruth of a Lie algebroid $A \to M$ and denote by $\nabla$ the $A$-connection on the 2-term complex $\partial: C \to E$ and $K \in \Omega^2(A, \Hom(E,C))$ the curvature term.

\begin{definition}\label{def:IMforms}
	An IM $k$-form on $A$ with values in the representation $\mathcal{E}$ is a triple $(\mathbb{D}, l, \theta)$, where
	 $\mathbb{D}: \Gamma(A) \longrightarrow \Omega^k(M,C)$,
	 $l: A \longrightarrow \wedge^{k-1}T^*M \otimes C$,
	 $\theta \in \Omega^k(M, E)$	
	are operators satisfying the Leibniz rule 
	\begin{equation}\label{Leibniz_ruth}
	\mathbb{D}(f\alpha) = f\mathbb{D}(\alpha) + df \wedge l(\alpha),
	\end{equation}
	(IM3) and (IM5) together with the additional IM equations
	\begin{align}
	\mathbb{D}([\alpha,\beta]) & =  L_{\nabla_\alpha} \mathbb{D}(\beta) - L_{\nabla_\beta} \mathbb{D}(\alpha)
	-  K_{\alpha,\beta} \circ \theta\label{Eq:IM1}\\
	l([\alpha,\beta])  & =  L_{\nabla_\alpha} l(\beta) - i_{\rho(\beta)} \mathbb{D}(\alpha)\label{Eq:IM2}\\
	L_{\nabla_\alpha}\theta & = \partial(\mathbb{D}(\alpha))\label{Eq:IM3}
	\end{align}
where $\alpha,\beta \in \Gamma(A)$.
\end{definition}

There is a close relationship between IM forms with values in ruth and IM forms with values in $\VB$-algebroids as we now explain. For a $\VB$-algebroid $\v \to E$, a splitting $\sigma: A \to \F(\v)$ of \eqref{linear_ses} gives rise to a ruth of $A$ on the graded vector bundle $C[1] \oplus E$ (see \cite{Gra-Met2, Arias-Crai2}). The complex $C \to E$ is the core complex $\partial: C \to E$, the $A$-connection on $E$ and the curvature $K \in \Omega^2(A, {\rm Hom}(E,C))$ are determined by the corresponding core sections of $\v$
(see Remark \ref{Bphi}):
\begin{align*}
\B(\nabla_{\alpha} c) & = [\sigma(\alpha), \B c], \,\,\, \nabla_{\alpha}|_E = \Delta_{\rho_{\v}(\sigma(\alpha))}^\top,\\
\B(K(\alpha_1, \alpha_2)) & = \sigma([\alpha_1,\alpha_2]) - [\sigma(\alpha_1), \sigma(\alpha_2)].
\end{align*}

\begin{lemma}\label{lemma:IMruth}
Let $\sigma: A \to \F(\v)$ be a splitting of \eqref{linear_ses} and consider the corresponding ruth on $\mathcal{E}=C[1]\oplus E$. One has that $(D,l,\theta)$ is an IM $k$-form on $A$ with values in $\v$ if and only if $(D\circ \sigma, l, \theta)$ is an IM $k$-form on $A$ with values in $\mathcal{E}$.
\end{lemma}

\begin{proof}
 The equivalence between equations \eqref{Eq:IM2}, \eqref{Eq:IM3} and (IM2), (IM4), respectively is an immediate consequence of the relationship between the $A$-connection and $\F(\v)$-connection on $\partial: C \to E$. The equivalence between \eqref{Eq:IM1} and (IM1) also uses the definition of $K$ and properties \eqref{eq:compatibility} of $D$. 
\end{proof}

 Given a ruth $(\partial, \Psi, \Omega)$ of $\G \toto M$ on $\E=C[1]\oplus E$, consider the semi-direct product $\V= \tar^*C \oplus \sour^*E \toto E$ and the corresponding $\VB$-algebroid $\v \to E$. We can apply the Lie functor to the natural horizontal lift $h: \G \to \F(\V)$ (since it preserves the identity and source maps) to obtain a splitting $\sigma: A \to \F(\v)$. Define $\mathrm{Lie}(\partial, \Psi, \Omega)$ as the corrresponding ruth $(\partial, \nabla, K)$ of $A$ on $\mathcal{E}=C[1]\oplus E$. We refer to \cite[Lemma~4.5]{Bra-Cab-Ort} for explicit formulas for $\nabla$ and $K$.   

\begin{theorem}
Let $\G \toto M$ be a source 1-connected Lie groupoid and $\mathcal{E}=C[1]\oplus E$ be a ruth of $G$ with structure operators $(\partial, \Psi, \Omega)$. There is a 1-1 correspondence between multiplicative forms $(\theta, \omega)$ with values in $\E$ and IM $k$-forms $(\mathbb{D},l, \theta)$ on $A$ with values in the ruth $\mathrm{Lie}(\partial, \Psi, \Omega)$. The correspondence is given by
\begin{align*}
\mathbb{D}(\alpha)(X_1,\dots, X_k) & = \left.\frac{d}{d\epsilon}\right|_{\epsilon=0} \Psi_{(\Fl^\epsilon_{\overrightarrow{\alpha}}(x))^{-1}} (\omega(T\Fl^\epsilon_{\overrightarrow{\alpha}}(X_1),\dots, \Fl^\epsilon_{\overrightarrow{\alpha}}(X_k)))\\
l(\alpha) & = i_{\overrightarrow{\alpha}}\omega|_M
\end{align*}
\end{theorem}

\begin{proof}
 Let $\V = \tar^*C \oplus \sour^*E \toto E$ be the $\VB$-groupoid corresponding to the ruth on $\E$ and consider its $\VB$-algebroid $\v=\mathrm{Lie}(\V)$. There is a chain of 1-1 correspondences coming from Lemmas \ref{lemma:rep_vb}, \ref{lemma:IMruth} and Theorem \ref{thm:main} 
 $$
 (\omega,\theta) \text{\, mult.} \leftrightarrow \vartheta = (\omega, \sour^*\theta) \in \Omega_{mult}^k(\G, \V) \leftrightarrow  (D,l,\theta) \leftrightarrow (\mathbb{D}=D\circ \sigma, l, \theta),
 $$
 where $\sigma: A \to \F(\v)$ is the natural splitting of $\v$. One only has to check the explicit formulas for the correspondence. Since $l(\alpha) = i_{\overrightarrow{\alpha}}\vartheta|_M = (i_{\overrightarrow{\alpha}}\omega|_M, 0)$, the formula for $l$ holds. Also, 
 $
 \mathbb{D}(\alpha) = L_{\Delta_{\overrightarrow{\sigma(\alpha)}}} \vartheta|_M.
 $
 Using  Proposition \eqref{prop:der_eq} and the fact that (see \cite{Bra-Cab-Ort})  
 $$
 \Fl^\epsilon_{\overrightarrow{\sigma(\alpha)}}(g,c,e) = (\Fl^\epsilon_{\overrightarrow{\alpha}}(g), \Psi_{\Fl^\epsilon_{\overrightarrow{\alpha}}(\tar(g))}(c) - \Omega_{(\Fl^\epsilon_{\overrightarrow{\alpha}}(\tar(g)), g)}(e), e) 
 $$
 one has that
 \begin{align*}
  (L_{\Delta_{\overrightarrow{\sigma(\alpha)}}} \vartheta)(X_1,\dots,X_k) &= \left.\frac{d}{d\epsilon}\right|_{\epsilon=0} \Fl_{\overrightarrow{\sigma(\alpha)}}^{-\epsilon}(\vartheta(T\Fl^\epsilon_{\overrightarrow{\alpha}}(X_1), \dots, T\Fl_{\overrightarrow{\alpha}}^\epsilon(X_k)))\\
  &  = \left.\frac{d}{d\epsilon}\right|_{\epsilon=0} [\Psi_{\Fl_{\overrightarrow{\alpha}}^{\epsilon}(x)^{-1}}(\omega(T\Fl_{\overrightarrow{\alpha}}^\epsilon(X_1), \dots, T\Fl_{\overrightarrow{\alpha}}^\epsilon(X_k)))\\
  & \hspace{-50pt} - \underbrace{\Omega_{(\Fl_{\overrightarrow{\alpha}}^{\epsilon}(x)^{-1}, \Fl_{\overrightarrow{\alpha}}^{\epsilon}(x))}(\theta(X_1,\dots, X_k))}_{f(\epsilon)} ],
 \end{align*}
where we have used that $\Fl^{-\epsilon}_{\overrightarrow{\alpha}}(\tar(\Fl_{\overrightarrow{\alpha}}^\epsilon(x))) = \Fl_{\overrightarrow{\alpha}}^{\epsilon}(x)^{-1}$. The result now follows from the fact that $\epsilon=0$ is a zero of order $\geq 2$ of $f(\epsilon)$. 
\end{proof}

\section{$\VB$-groupoid cohomology of differential forms}
In this section, we will introduce a cochain complex for which the cocycles in degree 1 are exactly the multiplicative forms with values in $\VB$-groupoids. Also, its cohomology will be a Morita invariant of the groupoid. This indicates that one should consider this complex as the appropriate setting to study connections on vector bundles over stacks. 

Let $B_\bullet \G$ be the nerve of the groupoid $\G \toto M$. It is the simplicial manifolds whose space of $p$-simplices is $B_p\G = \{(g_1, \dots, g_p) \in \G^p \,\, | \,\, \sour(g_i) = \tar(g_{i+1})\}$ and the face maps: $\partial_i : B_p\G \to B_{p-1}\G$, $i=0,\dots, p$ are defined by
$$
\partial_i(g_1, \dots, g_p) =
\begin{cases}
(g_2, \dots, g_p), & \text{ if } i=0,\\
(g_1, \dots, g_{i-1}, g_{i}g_{i+1},  g_{i+2}, \dots, g_{p}), & \text{ if } 1 \leq i \leq p-1,\\
(g_1, \dots, g_{p-1}), & \text{ if } i=p.
\end{cases}
$$
The differentiable cochain complex of $\G$ is $C^\bullet(\G) = C^{\infty}(B_{\bullet}\G)$ with the differential $\delta: C^{p-1}(\G) \to C^{p}(\G)$ given by 
$$
\delta = \sum_{i=0}^{p}(-1)^i \partial^*_i.
$$

For a $\VB$-groupoid $\V \toto E$ over $\G$, we define the complex $C^{p,q}(\V)$ of \textit{differential $\VB$-groupoid forms} as follows:
\begin{equation}
\begin{aligned}
\C^{0,q}(\V) & = \Omega^q(M, C)\\
\C^{p,q}(\V) & = \{\vartheta \in \Omega^q(B_p \G, \pr_1^*\V) \,\,| \,\,\ \widetilde{\sour} \circ \vartheta = \partial_0^*\theta, \, \, \theta \in \Omega^q(B_{p-1}\G, t^*E)\}, \,\,\, p \geq 1
\end{aligned}
\end{equation}
where $\pr_1: (g_1, \dots, g_p) \mapsto g_1$ is the projection on the first arrow and $t:(g_1,\dots, g_{p-1}) \mapsto \tar(g_1)$. The differential $\delta:\C^{p,q}(\V) \to C^{p+1,q}(\V)$ is defined as
\begin{equation}\label{differential}
\begin{aligned}
\delta \vartheta|_g & = - (\tar^*\vartheta) \bullet 0_g - 0_g \bullet (\sour^*\vartheta)^{-1}, \,\, p =0\\
\delta \vartheta & = -\partial_1^*\vartheta \bullet (\partial_0^*\vartheta)^{-1} + \sum_{i=2}^{p+1} (-1)^{i}\partial_i^*\vartheta , \,\, p \geq 1.
\end{aligned}
\end{equation}

It is straightforward to check $\vartheta \in \Omega^q(\G, \V)$ is multiplicative if and only if $\vartheta \in C^{1,q}(\V)$ and $\delta \vartheta = 0$. We shall postpone the proof that $(C^{\bullet, q}(\V), \delta)$ is a differential complex to \S \ref{proofs}.

\begin{example}\label{ex:cech}\em
 For $\V = \G \times \R$ with the trivial representation, one has that
 $$
 C^{p,q}(\G \times \R) = \Omega^q(B_p\G), \,\,\, \delta = \sum_{i=0}^p (-1)^i \partial_i^*.
 $$
 So, $(C^{p,q}(\G \times \R), \delta)$ recovers the Bott-Shulman complex of $q$-differential forms on Lie groupoids. Its cohomology is also known as the \textit{C\"ech cohomology of $\G$ with values in the sheaf of $q$-differential forms $\Omega^q$} \cite{Behrend}. 
\end{example}

\begin{example}\label{ex:2-term_rep}\em
 Let $\V = \sour^*E \oplus \tar^*C \toto E$ be the semidirect product of $\G$ with a ruth $(\partial,\Psi, \Omega)$ and consider  $\vartheta \in \Omega^q(B_p\G, \pr_1^*\V)$. It is straightforward to check that 
$\vartheta \in \C^{p,q}(\V)$ if and only if there exists $\omega \in \Omega^q(B_p\G, \tar^*C)$ and $\theta \in \Omega^q(B_{p-1}\G, \tar^*E)$ such that $\vartheta = (\omega, \partial_0^*\theta)$. In this case, using \eqref{Str}, one can check that 
$$
\delta \vartheta|_{(g_1, \dots, g_{p+1})} = -(\Omega_{(g_1,g_2)} \circ \left((\partial_0 \circ \partial_0)^*\theta\right) - \delta_C(\omega), \,\, \partial_0^*(\delta_E(\theta) + \partial \circ \omega)\,),
$$
where
$$
\delta_C(\omega)|_{(g_1,\dots, g_{p+1})} = \Psi_{g_1}\circ (\partial_0^*\omega) + \sum_{i=1}^{p+1}(-1)^i \partial_i^* \omega
$$
and similarly for $\delta_E$. For $q=0$, this is exactly the complex calculating the cohomology of $\G$ with values in the ruth $\E=C[1]\oplus E$, $H^\bullet(\G, \E)$. It is also important to note that, for $E=0$, the complex reduces to $(\Omega^q(B_\bullet \G, t^*C), \delta_C)$; this complex was introduced in \cite{Cab-Drum} in the context of van Est isomorphisms.
\end{example}

\begin{example}\em
For $q=0$, one has that $C^{p,0}(\V)= C_{\VB}^p(\V)$, the $\VB$-groupoid complex of $\V$ introduced in \cite{Gra-Met1}. This complex gives an intrisic model to the cohomology of $\G$ with values in representation up to homotopy. 
\end{example}

We shall now investigate the invariance of the cohomology $H^\bullet(\C^{\bullet,q}(\V))$ under Morita equivalence. A Lie groupoid morphism $\phi$ from $\G \toto M$ to $\G' \toto M'$ is a Morita map if it is fully faithful (i.e. the source and target maps define a good fibered product of manifolds $\G = \G' \times_{M'\times M'} (M \times M)$) and essentially surjective (the map $(g': y \leftarrow \phi(x), x) \mapsto \tar(g')$ is a surjective submersion from $\G' \times_{M'} M \to M'$). We refer to \cite{Hoyo} for details.

\begin{theorem}\label{thm:morita}
 Let $\V \toto E$ be a $\VB$-groupoid over $\G\toto M$ and $\phi: \G' \to \G$ a Morita map. Then the pull-back of differential forms $\phi^*: \C^{p,q}(\V) \to \C^{p,q}(\phi^*\V)$,
 $$
 (\phi^*\vartheta)(\underline{U}_1, \dots, \underline{U}_q) = \vartheta(T\phi(\underline{U}_1), \dots, T\phi(\underline{U}_q))), \,\,\,\,\,\underline{U}_i \in T(B_p\G).
 $$
 is a quasi-isomorphism.
\end{theorem}

By applying Theorem \ref{thm:morita} to the examples \ref{ex:cech} and \ref{ex:2-term_rep}, one obtains the following corollaries: 

\begin{corollary}\cite{Behrend}
 The C\"ech cohomology of $\G$ with values in the sheaf of $q$-differential forms $\Omega^q$ is a Morita invariant of $\G$.
\end{corollary}

\begin{corollary}\cite{Hoyo-Ort}
The cohomology of $\G$ with values in a 2-term representation up to homotopy is a Morita invariant.
\end{corollary}

We shall postpone the proof of Theorem \ref{thm:morita} to \S 5.4.

\section{The proofs}\label{proofs}
\subsection{Strategies}
Let us begin by explaining the general philosophy which will guide us through the proofs of Theorem \ref{thm:main} and \ref{thm:morita}. The same strategy used in \cite{Bur-Drum} to study multiplicative tensor fields on a Lie groupoid will give us the proper viewpoint to tackle the correspondence between multiplicative differential forms with values in $\VB$-groupoids and IM-forms with values in $\VB$-algebroids. The presence of coefficients encoded in the $\VB$-groupoid will force us to adapt (in a non-trivial way) most of the techniques used there as one can note by comparing this section with \cite[Section~4]{Bur-Drum}. 

The key point of the strategy is to treat a differential form $\vartheta \in \Omega^k(\G, \V)$ as a \textit{componentwise linear function} (see Definition in \S \ref{embedding} below)  on 
\begin{equation}\label{big_groupoid}
\bG= \underbrace{T\G \times_\G \dots \times_\G T\G}_{k-times}\times_\G \V^*.
\end{equation}
The space $\bG$ is a Lie groupoid and our aim is to establish a dictionary between the multiplicativity properties of $\vartheta$ and those of the corresponding function on $\bG$ as shown below:

 \begin{tabular}[t]{|c|c|}
 \hline
  \textbf{Differential forms} & \textbf{Componentwise linear}\\
  \textbf{with values in $\V$}& \textbf{functions on $\bG$}\\ \hline
  Multiplicativity & Cocycle equation on\\
  & the differentiable cochain complex\\ \hline
  Operations $L_\Delta, i_{\overrightarrow{\alpha}}$,  & Lie derivative along special\\
  restriction to the units & right-invariant vector fields\\ \hline
  Infinitesimal components $(D,l,\theta)$ &  Infinitesimal cocyle evaluated \\
  & on special sections of $\bA=\mathrm{Lie}(\bG)$\\ \hline
  IM-equations & Infinitesimal cocycle equation on the \\
  & Chevalley-Eilenberg complex \\ \hline
  Complex of differential  & Subcomplex of the cochain\\
  $\VB$-groupoid forms &   differentiable complex\\ \hline
 \end{tabular}
\vspace{10pt}

 Once this is settled, Theorem \ref{thm:main} will be a simple consequence of the 1-1 correspondence between cocycles on the differentiable cohomology of a (1-connected) Lie groupoid and cocycles on the Chevalley-Eilenberg of its Lie algebroid. Also,
 Theorem \ref{thm:morita} will follow from arguments similar to those used in \cite{Hoyo-Ort} to prove the Morita invariance of the $\VB$-groupoid cohomology of $\V$.

\subsection{General embedding trick}\label{embedding}
In this first section we recall how tensor fields can be embedded as functions on fiber products. We also extend the formulas on \cite[Section~4.1]{Bur-Drum} relating Lie derivatives of tensor fields and Lie derivatives of their corresponding componentwise linear functions along lifted vector fields to a more general setting. These formulas will serve as basis for most of the calculations done in the proof of Theorem \ref{thm:main}. 

Given vector bundles $E_i \to M$, denote by $\times_M E_i$ the $k$-fold fiber product $E_1 \times_M \dots \times_M E_k$. Consider the embedding
\begin{equation}\label{F_def}
\begin{array}{ccc}
\c: \,\,\,\Gamma(E_1^* \otimes \dots \otimes E_k^*) & \hookrightarrow & \hspace{-30pt}C^\infty(\times_M E_i)\vspace{10pt}\\
\tau = \mu_1 \otimes \dots \otimes \mu_k & \mapsto &  \c_\tau = \ell_{\mu_1} \circ \pr^1 \cdots \ell_{\mu_k} \circ \pr^k,
\end{array}
\end{equation}
 where $\pr^i: E_1 \times_M \dots \times_M E_k \to E_i$ is the natural projection. We shall refer to the functions in the image as \textit{componentwise linear functions}. When $E_1 = \dots = E_k = E$, we see $\Gamma(\Lambda^k E^*)$ inside $\Gamma(\otimes^k E^*)$ as usual:
$$
\mu_1 \wedge \dots \wedge \mu_k \mapsto \sum_{\sigma \in S_k} \mathrm{sgn}(\sigma) \mu_{\sigma(1)} \otimes \dots \otimes \mu_{\sigma(k)}
$$
\begin{remark}{\em
The space $C^\infty(\times_M E_i)$ has a $C^\infty(M)$-module structure defined by multiplication of functions (see $C^\infty(M)$ inside it by pulling back functions under the projection $\times_M E_i \to M$). The map $\c$ is a morphism of $C^\infty(M)$-modules.  
}
\end{remark}

There is a notion of tensor product for derivations that we now describe. Let $\Delta_i: \Gamma(E_i^*) \to \Gamma(E_i^*)$, $i=1, \dots, k$, be derivations having the same symbol $\sharp(\Delta) = X \in \frakx(M)$. The \textit{tensor product derivation} is defined as
 \begin{align*}
 \otimes_1^k \Delta_i = \Delta_1\otimes \dots \otimes \Delta_k&: \Gamma(E_1^* \otimes \dots \otimes E_k^*) \to \Gamma(E_1^* \otimes \dots \otimes E_k^*)\\
 (\otimes_1^k\Delta_i) (\mu_1\otimes \dots \otimes \mu_k) & = \sum_{i=1}^k \mu_1 \otimes \dots \otimes \Delta_i(\mu_1) \otimes \dots \otimes \mu_k
 \end{align*}
 
\begin{example}\em
   Let $\tau \in \Omega^k(M, E)$ and $\Delta: \Gamma(E) \to \Gamma(E)$ be a derivation with symbol equal to $X \in \frakx(M)$. One has that 
 \begin{equation}
L_\Delta(\tau) = (\underbrace{\Lie_X \otimes \dots \otimes \Lie_X}_{k-times} \otimes \Delta) (\tau), 
 \end{equation}
 where $L_\Delta$ is the operator \eqref{der_Lie}.
  \end{example}
  
\begin{example}\em\label{tensor_example}
Let $\tau \in \Omega^k(M, \Lambda^l A)$, for a Lie algebroid $(A, [\cdot, \cdot], \rho)$. In this case, 
$$
(\underbrace{\Lie_{\rho(\alpha)} \otimes \dots \otimes \Lie_{\rho(\alpha)}}_{k-times} \otimes \underbrace{[\alpha, \cdot] \otimes \dots \otimes [\alpha, \cdot]}_{l-times})(\tau)= a \cdot \tau,
$$
where $\cdot$ is the natural action of $\Gamma(A)$ on $\Omega^k(M, \Lambda^lA)$ defined as
$$
\alpha \cdot (\xi \otimes \frakx) = (\Lie_{\rho(\alpha)} \xi) \otimes \frakx + \xi \otimes [\alpha, \frakx],
$$
for the Schouten-Nijenhuis bracket $[\cdot, \cdot]$ on $\Lambda^\bullet A$.
\end{example}

Our next result concerns some equivariance properties of the map $\c$. It should be seen as a generalization of Proposition~4.1 in \cite{Bur-Drum} (when applied to Example \ref{tensor_example}).
 
 \begin{proposition}
 Let $\Delta_i: \Gamma(E_i^*) \to \Gamma(E_i^*)$ be derivations having the same symbol, $W_i \in \frakx(E_i)$ the corresponding linear vector fields and $\varphi_j \in \Gamma(E_j)$. One has that
 \begin{align}
 \label{der_inv} \Lie_{(W_1, \dots, W_k)} \c_\tau & = \c_{\otimes_1^k\Delta_i \tau}  \\
 \label{contraction_inv}  \Lie_{(0,\dots, \varphi_j^\uparrow, \dots, 0)} \c_\tau & = \c_{\tau(\dots, \varphi_j, \dots)} \circ \gamma_{(j)},
 \end{align}
 where $\gamma_{(j)}: \prod_i E_i \to \prod_{i \neq j} E_i$ is the forgetful projection.
\end{proposition} 
 \begin{proof}
  The result follows \eqref{F_def}, the Leibniz rule for Lie derivatives and
  $$
  \Lie_{W_j} \ell_{\mu_j} = \ell_{\Delta_j(\mu_j)}, \,\,\, \Lie_{\varphi_j^\uparrow} \ell_{\mu_j} = \<\mu_j, \varphi_j\> \circ p_j,
  $$
  where $\mu_j \in \Gamma(E_j^*)$ and $p_j: E_j \to M$ is the projection.
 \end{proof}

 \begin{corollary}\label{cor:action}
  Let $\tau \in \Omega^k(M, E)$ and $\Delta: \Gamma(E) \to \Gamma(E)$ be a derivation with symbol $X \in \frakx(M)$. One has that
  \begin{align*}
  \Lie_{(X^{T,k}, W)} \c_\tau & = \c_{L_\Delta \tau}, \\
  \Lie_{(Y^{\uparrow, k}_{(j)},0)} \c_\tau & = (-1)^{j-1} \c_{i_Y \tau} \circ \gamma_{(j)}, \,\, j=1, \dots, k,\\
  \Lie_{(0, \varphi^\uparrow)} \c_\tau & = \c_{\<\varphi, \tau\>} \circ \gamma_{(k+1)},
  \end{align*}
  where $(X^{T,k}, W), Y^{\uparrow,k}_{(j)} \in \frakx(\times_M^k TM)$ are the vector fields
  \begin{align*}
  (X^{T, k}, W)(Z_1, \dots, Z_k,e) & = (X^T(Z_1), \dots, X^T(Z_k), W(e)),\\
  Y^{\uparrow,k}_{(j)}(Z_1, \dots, Z_k) & = (0_{Z_1}, \dots,0_{Z_{j-1}}, Y^\uparrow(Z_j), 0_{Z_{j+1}}, \dots, 0_{Z_k}),
  \end{align*}
  $(\cdot)^T, (\cdot)^\uparrow$ is the tangent and vertical lift, respectively, $W$ is the linear vector field associated to $\Delta$.
 \end{corollary}

\subsection{Proof of Theorem \ref{thm:main}}\label{proof1}
\subsubsection{From global to infinitesimal}
Given a $\VB$-groupoid $\V \toto E$ over $\G \toto M$, there is a natural groupoid $\bG \toto \bM$, where
$$
\bG= \times_\G^k T\G \times_\G \V^*, \,\,\, \bM = \times_M^k TM \times_M C^*
$$
and the groupoid structure maps are defined as the fiber product of each of the corresponding maps on $T\G$ and $\V^*$. \footnote{Note that we are using the fact that the projection maps $(T\G \toto TM) \to (\G \toto M)$ and $(\V^* \toto C^*) \to (\G \toto M)$ fulfill the conditions for the fiber product to have a Lie groupoid structure. We refer to the appendix of \cite{Bur-Cab-Hoy} for a detailed discussion regarding fiber products of Lie groupoids.} 

\begin{proposition}
 $\vartheta \in \Omega^k(\G, \V)$ is multiplicative if and only if $\c_\vartheta$ is a cocycle on the differentiable cohomology of $\bG$.
\end{proposition}

\begin{proof}
 It is straightforward to see that $\vartheta$ multiplicative implies that $\c_{\vartheta}$ is a cocycle. Conversely, let us assume that $\c_\vartheta$ is a cocycle. First, note that $\c_\vartheta|_\bM = 0$ implies that
  $
  \vartheta|_M \in \Omega^k(M, E)
  $
  under the decomposition $\V|_M = E\oplus C$. So, define $\theta \in \Omega^k(M, E)$ as $\theta:=\vartheta|_M$. We claim that $\widetilde{\sour} \circ \vartheta = \sour^*\theta$ and $\widetilde{\tar}\circ \vartheta = \tar^*\theta$. Indeed, for $\varphi \in E_{\sour(g)}^*$ and $U_1, \dots, U_k \in T_{g} \G$, it follows from the multiplicativity of both $\c_\vartheta$ and the pairing $\<\cdot, \cdot\>$ that
  \begin{align*}
   \< \widetilde{\sour}(\vartheta(U_1, \dots, U_k)), \varphi\> & = \< \widetilde{\sour}(\vartheta(U_1, \dots, U_k)), \varphi - \partial^*(\varphi)\>\\
   & = \<\vartheta(U_1, \dots, U_k)\bullet \widetilde{\sour}(\vartheta(U_1, \dots, U_k)), 0_{g_1} \bullet (\varphi - \partial^*(\varphi)) \>\\
   & = \c_{\vartheta}(U_1, \dots, U_k, 0_{g_1} \bullet (\varphi - \partial^*(\varphi)))\\
  & = \c_\vartheta(T\sour(U_1), \dots, T\sour(U_k), \varphi - \partial^*(\varphi))\\
  & = \<\theta(T\sour(U_1), \dots, T\sour(U_k)), \varphi\>.
  \end{align*}
A similar argument proves that $\widetilde{\tar}\circ \vartheta = \tar^*\theta$. Finally, to prove compatibility with the multiplication, it follows directly from the cocycle equation for $\c_\vartheta$ that
\begin{align*}
 \<\vartheta(U_1 \bullet V_1, \dots, U_k \bullet V_k), \psi_1 \bullet \psi_2\> = \< \vartheta(U_1, \dots, U_k) \bullet \vartheta(V_1, \dots, V_k), \psi_1 \bullet \psi_k\>,
\end{align*}
for any $(V_1, \dots, V_k) \in \times^k T_{g_2}\G$ composable with $(U_1, \dots, U_k)$ and $\psi_1 \in \V^*_{g_1}, \psi_2 \in \V^*_{g_2}$ a composable pair. The result now follows from the fact that any $\psi \in \V^*_{g_1g_2}$ can be written as a produt $\psi = \psi_1 \bullet \psi_2$.
\end{proof}

The Lie algebroid of $\bG$ is $\bA \to \bM$ 
$$
\bA = \times_A^k TA \times_A \v^*_A,
$$
with Lie algebroid structure determined componentwise by $TA \to TM$ and $\v_A^* \to C^*$. Note that we are using identifications \eqref{dual_identification} and $TA \cong \mathrm{Lie}(T\G)$ (see \cite{Mckz-Xu} for a discussion of tangent groupoids and tangent algebroids). 

It follows from the general theory of double vector bundles that $\Gamma(\bM, \bA)$ is generated as a $C^\infty(\bM)$-module (see \cite[Proposition~2.2]{Mac-doubles}) by $\chi_\eta, \B \beta_{(j)}, \B \varphi: \bM \to \bA$ defined as follows:
\begin{align}\label{gen_set}
\begin{split}
\chi_\eta(X_1,\dots, X_k, \mu) & = (T\alpha(X_1), \dots, T\alpha(X_k), \eta^\top(\mu)),\\
\B \beta_{(j)}(X_1, \dots, X_k, \mu) & = (T0(X_1), \dots, B\beta(X_j), \dots, T0(X_k), 0_\mu),\\
\B\varphi(X_1, \dots, X_k,\mu) & = (T0(X_1), \dots, T0(X_k), \B\varphi(\mu)),
\end{split}
\end{align}
where $\eta \in \Gamma_{lin}(E,\v)$, $\alpha = \pr(\eta), \,\beta \in \Gamma(A)$ and $\varphi \in \Gamma(E^*)$. Using \eqref{eq:dual_vb}, one can see that the Lie algebroid structure of $\bA \to \bM$ is given in terms of the generators as follows (see Corollary \eqref{cor:action} for notation):
\begin{equation}\label{big_algebroid}
\begin{aligned}
 \noindent [\chi_{\eta_1}, \chi_{\eta_2}]  = \chi_{[\eta_1, \eta_2]}, & \,\,\,\,\,\, [\chi_\eta, \B \beta_{(j)}] = \B [\alpha, \beta]_{(j)}\\
  [\chi_\eta, \B \varphi]  = \B(\nabla_\eta^\top \varphi), & \,\,\,\,\,\,\,[\B(\ast), \B(\ast')] = 0\\
 \rho_\bA(\chi_\eta)  = (\rho(\alpha)^{T,k}, W_{\nabla_\eta}) & ,  \,\,\ \rho_\bA(\B\beta_{(j)}) = (\rho(\alpha)^{\uparrow, k}_{(j)},0), \,\,\, \rho_\bA(\B\varphi) = (0, \partial^*(\varphi)^\uparrow).
 \end{aligned}
\end{equation}



\begin{proposition}\label{prop:inf_comp}
Let $\vartheta \in \Omega^k(\G, \V)$ be a multiplicative $k$-form. There exists $D: \Gamma_{lin}(E, \v) \to \Omega^k(M, C)$, $l: A \to \wedge^{k-1}T^*M \otimes C$ and $\theta \in \Omega^k(M,E)$ such that
$$
D(\mathcal{B}\Phi) = -\Phi \circ \theta, \,\,\,\, D(f\eta) = fD(\eta) + df \wedge l(\pr(\eta)), 
$$
for $\Phi \in \mathrm{Hom}(E,C)$, $f \in C^\infty(M)$, satisfying
\begin{align*}
D(\eta) & = L_{\Delta_\eta}(\vartheta)|_M\\
l(\alpha) & = i_{\overrightarrow{a}} \vartheta|_M\\
\theta & = \vartheta|_M.
\end{align*}
\end{proposition}

\begin{proof}
Consider the multiplicative function $\c_\vartheta \in C^\infty(\bG)$ and the corresponding Lie algebroid cocycle $\mathrm{Lie}(\c_\vartheta) \in \Gamma(\bA^*)$. By evaluating $\mathrm{Lie}(\c_\vartheta)$ on $\chi \in \Gamma(\bA)$,
\begin{equation}\label{eq:inf_global}
\<\mathrm{Lie}(\c_\vartheta), \chi\> = \Lie_{\overrightarrow{\chi}}\c_\vartheta|_\bM,
\end{equation}
with $\chi$ varying in the set of generators \eqref{gen_set}, one obtains maps from the space of parameters $\Gamma_{lin}(E,\v)$, $\Gamma(A)$ and $\Gamma(E^*)$ to $C^\infty(\bM)$. We will show that $D,l,\theta$ will appear as refinements of these maps. First note that it follows from \cite[Eq.~(45), \S 9.7]{Mckz2}, Propositions \ref{right_core} and \ref{right_adjoint} that 
\begin{align*}
 \overrightarrow{\chi_\eta} & = (\overrightarrow{\alpha}^{T,k}, \overrightarrow{\eta}^\top), \,\,\,\, \overrightarrow{\B\beta_{(j)}}  = (\overrightarrow{\beta}^{\uparrow,k}_{(j)},0), \,\,\,\,\,\,
 \overrightarrow{\B \varphi}  = (0, \varphi_R^\uparrow).
\end{align*}
From Corollary \eqref{cor:action}, 
\begin{align*}
 \Lie_{\overrightarrow{\chi_\eta}}\c_\vartheta = \c_{L_{\Delta_\eta}\vartheta}, \,\,\, \Lie_{\overrightarrow{\B\beta_{(j)}}}\c_\vartheta = (-1)^{j-1} \c_{i_{\beta} \vartheta} \circ \pi_{(j)}, \,\,\,\,
 \Lie_{\overrightarrow{\B \varphi}} \c_\vartheta = \c_{\<\varphi_R, \vartheta\>}\circ \pi_{(k+1)},
 \end{align*}
where $\pi_{(j)}: \times_\G^k T\G \times_\G \V^* \to \times_\G^{k-1} T\G \times_\G \V^*$ forgets the $j$-th component, $j=1,\dots, k$, and $\pi_{(k+1)}: \times_\G^k T\G \times_\G \V^* \to \times_\G^k T\G$ forgets the component on $\V^*$. So, by restricting to the units $\bM \subset \bG$ and using \eqref{eq:inf_global}, one notes that there exists $D: \Gamma_{lin}(E, \v) \to \Omega^k(M, C)$, $l_j: \Gamma(A) \to \Omega^{k-1}(M,C), \, j=1,\dots, k,$ and $r:\Gamma(E^*) \to \Omega^k(M)$ satisfying 
\begin{align}\label{eq:liner_inf}
\begin{split}
\c_{D(\eta)}  & =\<\mathrm{Lie}(\c_\vartheta), \chi_\eta\> = \c_{L_{\Delta_\eta} \vartheta} |_\bM  \\
\c_{l_j(\beta)} \circ \gamma_{(j)} & = \<\mathrm{Lie}(\c_{\vartheta}), \B \beta_{(j)}\> = (-1)^{j-1}\c_{i_{\beta} \vartheta} \circ \pi_{(j)}|_\bM\\
\c_{r(\varphi)} \circ \gamma_{(k+1)}  & = \<\mathrm{Lie}(\c_{\vartheta}), \B \varphi\>  = \c_{\<\varphi_R, \vartheta\>} \circ \pi_{(k+1)}|_\bM
\end{split}
\end{align}
Since \eqref{F_def} is injective, it follows that $ D(\eta)=L_{\Delta_\eta}(\vartheta)|_M,  \,\,\,  l_j(\beta) = (-1)^{j-1}i_\beta \vartheta |_M, \,\,\, r(\varphi) = \<\varphi_R, \vartheta\>|_M$. Note that $l_j$ and $r$ are $C^\infty(M)$-linear. Also, if we define $l=l_1$, then $l_j=(-1)^{j-1}l$ \footnote{The skew-symmetry of $\vartheta$ gives the relationship between the $l_j$'s. In general, for a tensor $\tau \in \Gamma(\otimes^k T^*\G \otimes \V)$, one gets $k$ non-related maps $l_j: A \to \otimes^{k-1} T^*\G \otimes \V$ by contracting $\tau$ with $\overrightarrow{\beta}$ on the $j$-th factor.}\,. Also,  
\begin{align*}
 \<\varphi_R(g),\vartheta(U_1, \dots, U_k)\> & = \<\varphi(\tar(g))\bullet 0_g, \vartheta(T\tar(U_1)\bullet U_1, \dots, T\tar(U_k) \bullet U_k)\>\\
 & = \< \varphi(\tar(g)), \,\theta(T\tar(U_1),\dots, T\tar(U_k)\> 
 \end{align*}
 In other words, $\<\varphi_R,\vartheta\> = \tar^*\<\varphi, \theta\>$. So, if one sees $r$ as an element of $\Omega^k(M, E)$, one has that $r=\theta$. The properties of $D$ follows from the following facts:
 $$
 L_{\Delta_{\B \Phi}}(\vartheta) = - \Phi \circ \vartheta, \,\, L_{\Delta_{f\eta}} \vartheta = (\tar^*f) L_{\Delta_{\eta}}(\vartheta) + (\tar^*df)\wedge i_{\overrightarrow{\alpha}}\vartheta.
 $$
 The minus sign on $L_{\Delta_{\B \Phi}}$ comes from the relationship between $\Delta_{\B \Phi}$ and $\Delta_{\B \Phi}^\top$ (see Remark \ref{der_linear}).
\end{proof}

\subsubsection{From infinitesimal to global}
Given a triple $(D, l, \theta)$ satisfying the compatibilities \eqref{eq:compatibility}, we construct a function 
$
\Upsilon: \times_A^k TA \times_A \v_A^* \to \R
$
as follows:
 for 
 $$
 \xy
{\ar@{|->}_{}(0,25)*++{\chi_j \in TA};(22,25)*++{a \in A}}\\
{\ar@{|->}_{}(0,25)*++{};(0,15)*++{}}\\
{\ar@{|->}_{}(0,15)*++{X_j \in TM};(22,15)*++{x \in M}}\\
{\ar@{|->}_{}(22,25)*++{};(22,15)*++{}}\\
\endxy
\,\,\,\,\,\,\,\text{ and } \,\,\,\,\,\,\,\,
 \xy
{\ar@{|->}_{}(0,25)*++{\phi \in \v_A^*};(20,25)*++{a \in A}}\\
{\ar@{|->}_{}(0,25)*++{};(0,15)*++{}}\\
{\ar@{|->}_{}(0,15)*++{\mu \in C^*};(20,15)*++{x \in M,}}\\
{\ar@{|->}_{}(20,25)*++{};(20,15)*++{}}\\
\endxy
\vspace{25pt}
$$
$j=1, \dots, k$, choose $\zeta \in \Gamma_{lin}(E,\v)$ such that $\zeta^\top(\mu)$ projects over $a$ under the projection $\v_A^* \to A$ and
let $\xi \in E_x^*$ and $a_j \in A_x$ be elements determined by
$$
\phi = \zeta^\top(\mu) +_{C^*} (0_\mu +_A \overline{\xi}), \,\,\,\,\chi_j = T\alpha(X_j) +_{TM} (T0(X_j) +_A \overline{a_j}), \,\,\, \alpha = \pr(\zeta).
$$
Define
\begin{align}\label{Flinear_def} 
\begin{split}
\Upsilon(\chi_1, \dots, \chi_k, \phi) & = \<\mu, D(\zeta)(X_1, \dots, X_k)\>\\
 & \hspace{-50pt}+  \<\mu,\sum_{j=1}^k (-1)^{j-1} l(a_j)(X_1, \dots, \widehat{X_{j}}, \dots, X_k)\> + \<\xi, \theta(X_1,\dots, X_k)\>.
\end{split}
\end{align}

Note that $\Upsilon$ is skew-symmetric with respect to the $TA$ components.
\begin{proposition}
 $\Upsilon$ is a well-defined componentwise linear function (i.e. there exists $\tau \in \Gamma(\wedge^k T^*A \otimes \v)$ such that $\Upsilon = \c_\tau$). Moreover, $\Upsilon$ is fiberwise linear with respect to vector bundle structure $\bA \to \bM$. As an element of $\Gamma(\bA^*)$, one has that
 \begin{equation}
 \begin{aligned}\label{linear_comp}
 \<\Upsilon, \chi_\eta\> & = \c_{D(\eta)}  , \,\, \<\Upsilon, \B\beta_{(j)}\> = (-1)^{j-1}\c_{l(\beta)}\circ \gamma_{(j)} \\ \<\Upsilon, \B\varphi\> & = \c_{\<\varphi, \theta\>} \circ \gamma_{(k+1)}
 \end{aligned}
 \end{equation}
 as functions on $\bM$.
\end{proposition}

\begin{proof}
 Let $\{\alpha_r\}_{r=1, \dots, \mathrm{rank}(A)}$, $\{\varphi_s\}_{s=1, \dots, \mathrm{rank}(E)}$ be local frames of $A$ and $E^*$, respectively. If $\eta_r \in \Gamma_{lin}(E, \v)$ are local sections such that $\pr(\eta_r) = \alpha_r$, then $\{\eta_r^\top, \B \varphi_s\}$ is a local frame for $\v_A^* \to C^*$ (see \cite[Prop.~2.2]{Mac-doubles}). Write
 $$
 \chi_j = t^r\cdot T\alpha_r(X_j) +_{TM} h_{j}^{r}\cdot \B \alpha_r(X_j), \,\,\, \phi = t^r\cdot \eta^\top_r(\mu) +_{C^*} g^s \cdot \B\varphi_s,
 $$
 for $t^r, h_j^r, g^s \in \R$ and $a = t^r \alpha_r(x)$ (here $\cdot$ stands for the multiplication by scalars on $TA \to TM$ and $\v_A^* \to C^*$ respectively). Now, for any linear section $\zeta \in \Gamma_{lin}(E, \v)$ 
 $$
 p_*: \zeta^\top(\mu) \mapsto a \Leftrightarrow \zeta^\top = (f^r \circ p_*)   \cdot \eta_r^\top, \,\,\,\text{ for } f^r \in C_{loc}^\infty(M) \text{ such that } f^r(x) = t^r,
 $$
 where $p_*: \v_A^* \to A$ is the vector bundle projection. So, $\xi = g^s \varphi_s(x)$ and, using that
 $$
 T(f^r \alpha_r)(X_j) = f^r(x) \cdot T\alpha_r(X_j) +_{TM} (\Lie_{X_j}f)(x)\cdot \B\alpha_r(X_j),
 $$
 one gets that
 $$
 a_j = (h_j^r - (\Lie_{X_j}f^r)(x)) \alpha_r(x).
 $$
 Using the Leibniz rule for $D$, one finally obtains the following local expression for $\Upsilon$:
 \begin{align*}
 \Upsilon(\chi_1, \dots, \chi_k, \phi)& = \<\mu, t^rD(\eta_r)(X_1,\dots, X_k)\\ 
 & \hspace{-50pt}+ \sum_{j=1}^k (-1)^{j-1} h^r_j \, l(\alpha_r)(X_1, \dots, \hat{X_j}, \dots, X_k)\> +  g^s \<\varphi_s, \theta(X_1,\dots, X_k)\>.
 \end{align*}
The dependence on $t^r, h_j^r, g^s$ of the local expression of $\Upsilon$ shows that it is fiberwise linear on $\bA \to \bM$. At last, the equalities \eqref{linear_comp} follow from the definition of $\Upsilon$ and 
\begin{equation*}
 (\chi_1, \dots, \chi_k, \phi) = 
 \begin{cases}
  \chi_\eta(X_1, \dots, X_k, \mu) \Leftrightarrow a_i=0, \, \xi=0, \text{ for } \zeta =\eta\\
  \B\beta_{(j)}(X_1, \dots, X_k, \mu)  \Leftrightarrow a_i = \delta_i^j \beta, \,\xi=0, \text{ for } \zeta =0\\
  \B \varphi(X_1, \dots, X_k,\mu) \Leftrightarrow a_i=0,\, \xi = \varphi(x), \text{ for } \zeta = 0, 
 \end{cases}
 \end{equation*}
where $\delta^j_i$ is the Kronecker delta.
 \end{proof}

\begin{proposition}\label{prop:IM_eq}
The section $\Upsilon \in \Gamma(\bA^*)$ given by \eqref{Flinear_def} is a Lie algebroid cocycle if and only if $(D,l,\theta)$ is an IM $k$-form with values in $\v$.
\end{proposition}

\begin{proof}
The proof follows closely the proof of Proposition 4.12 in \cite{Bur-Drum}.  By definition, $\Upsilon$ is a cocycle if and only the equation 
 \begin{equation}\label{eq:cocycle}
 \<\Upsilon,[\chi_1, \chi_2]\> = \Lie_{\rho_\bA(\chi_1)} \< \Upsilon, \chi_2\> - \Lie_{\rho_\bA(\chi_2)} \<\Upsilon, \chi_1\>
 \end{equation}
 is fulfilled for any $\chi_1, \chi_2 \in \Gamma(\bA)$. It suffices to consider $\chi_1, \chi_2$ varying on the set of generators \eqref{gen_set}. In these cases, one uses equations \eqref{linear_comp} and \eqref{big_algebroid} on each of the six possible cases to show the equivalence of \eqref{eq:cocycle} to the set of IM equations.
 \smallskip
 
 \paragraph{\bf Equation (IM1):} Let $\eta_1, \eta_2 \in \Gamma_{lin}(E, \v)$ and consider $\chi_1= \chi_{\eta_1}, \, \chi_2 = \chi_{\eta_2}$. In this case, one can see that \eqref{eq:cocycle} is equivalent to
 \begin{align*}
  \c_{D([\eta_1, \eta_2])} & = \Lie_{(\rho(\alpha_1)^{T,k}, W_{\nabla_{\eta_1}})} \c_{D(\eta_2)} - \Lie_{(\rho(\alpha_2)^{T,k}, W_{\nabla_{\eta_2}})} \c_{D(\eta_1)}\\
   & = \c_{L_{\nabla_{\eta_1}} D(\eta_2) - L_{\nabla_{\eta_2}} D(\eta_1)},
 \end{align*}
 where the last equality follows from Corollary \ref{cor:action}.
 \smallskip
 
 \paragraph{\bf Equations (IM2) and (IM4):} Let $\eta \in \Gamma_{lin}(E,\v)$ and consider $\chi_1=\chi_{\eta}$, $\chi_2 = \B \beta_{(j)}$. In this case, \eqref{eq:cocycle} reduces to 
 \begin{align*}
 \<\Upsilon, \B([\alpha, \beta])_{(j)}\> & = (-1)^{(j-1)} \Lie_{(\rho(\alpha)^{T,k}, W_{\nabla_{\eta}})}(\c_{l(\beta)}\circ \gamma_{(j)}) - \Lie_{(\rho(\beta)^{\uparrow, k}_{(j)},0)} \c_{D(\eta)}\\
 & = (-1)^{(j-1)} (\Lie_{(\rho(\alpha)^{T,k-1}, W_{\nabla_{\eta}})} \c_{l(\beta)}) \circ \gamma_{(j)} - \Lie_{(\rho(\beta)^{\uparrow, k}_{(j)},0)} \c_{D(\eta)}.
\end{align*}
Now, using Corollary \ref{cor:action}, one can show that the last equation is equivalent to
\begin{align*}
 (-1)^{j-1}\c_{l([\alpha,\beta])} \circ \gamma_{(j)} = (-1)^{j-1}(\c_{L_{\nabla_\eta} l(\beta) - i_{\rho(\beta)}D(\eta)})\circ \gamma_{(j)}
\end{align*}
 Similarly, if $\chi_2 = \B \varphi$, for $\varphi \in \Gamma(E^*)$, one obtains that
$$
\<\nabla_\eta^\top \varphi, \theta\>  = \Lie_{\rho(\alpha)} \<\varphi, \theta\> - \<\varphi, \partial(D(\eta))\> \Leftrightarrow \<\varphi, L_{\nabla_\eta} \theta\> = \< \varphi, \partial(D(\eta))\>.
$$
\smallskip

\paragraph{\bf Equation (IM3):} 
It follows exactly as in the proof of Equation (IM4) in \cite{Bur-Drum}.

\paragraph{\bf Equation (IM5):} Let now $\chi_1= \B\beta_{(j)}$, $j=1,\dots, k$, and $\chi_2 = \B\varphi$. As $[\B\beta_{(j)},\B\varphi ]=0$, equation \eqref{eq:cocycle} reduces to 
\begin{align*}
0 & = \Lie_{(\rho(\beta)^{\uparrow,k}_{(j)},0)} (\c_{\<\varphi, \theta\>}\circ \gamma_{(k+1)}) - (-1)^{j-1} \Lie_{(0, \partial^*(\varphi)^\uparrow)} (\c_{l(\beta)}\circ \gamma_{(j)})\\
 & = (\Lie_{\rho(\beta)^{\uparrow,k}_{(j)}}\c_{\<\varphi, \theta\>}) \circ \gamma_{(k+1)} -  (-1)^{j-1} (\Lie_{(0, \partial^*(\varphi)^\uparrow))} \c_{l(\beta)}) \circ \gamma_{(j)}\\
& = (-1)^{j-1}\c_{\<\varphi, i_{\rho(\beta)}\theta\>} \circ \gamma_{(j)}\circ \gamma_{(k+1)}- (-1)^{j-1}\c_{\<\varphi, \partial(l(\beta))\>}\circ \gamma_{(k)}\circ \gamma_{(j)}.
 \end{align*}
 The equation now follows from the fact that $\gamma_{(j)}\circ \gamma_{(k+1)}$ and $\gamma_{(k)}\circ \gamma_{(j)}$ are the same projection from $\times_M^{k} TM \times_M C^*$ from $\times_M^{k-1} TM$ which forgets the $j$-th component on $TM$ and the last component on $C^*$.
 \end{proof}

\begin{proof}[Proof of Theorem \eqref{thm:main}]
\noindent
\paragraph{\bf Differentiation:} Let $\vartheta \in \Omega^k(\G, \V)$ be a multiplicative $k$-form. From Proposition \ref{prop:inf_comp}, there exists $D: \Gamma_{lin}(E,\v) \to \Omega^k(M,C)$,  $l: A \to \wedge^{k-1} T^*M \otimes C$ and $\theta \in \Omega^k(M,E)$ satisfying \eqref{eq:inf_comp}. If $\Upsilon \in \Gamma(\bA^*)$ is the element defined by \eqref{Flinear_def} using the triple $(D,l,\theta)$, then $\Upsilon=\mathrm{Lie}(\c_{\vartheta})$. Indeed, the equality follows from comparing the values of $\Upsilon$ and $\mathrm{Lie}(\c_\vartheta)$  on the set of generators of $\Gamma(\bA)$ as given by equations \eqref{eq:liner_inf} and \eqref{linear_comp}, respectively. As $\mathrm{Lie}(\c_\vartheta)$ is a cocycle on $\bA$, it follows from Proposition \ref{prop:IM_eq} that $(D,l,\theta)$ satisfies the IM-equations.
\smallskip

\paragraph{\bf Integration:} Let $(D,l,\theta)$ be an IM $k$-form on $A$ with values on $\v=\mathrm{Lie}(\V)$ and define $\Upsilon$ as \eqref{Flinear_def}. By Proposition \ref{prop:IM_eq},  $\Upsilon \in \Gamma(\bA^*)$ is a Lie algebroid cocycle. As $\G$ is 1-connected, so is $\bG$ \footnote{We are using the well-known fact that a source fiber of a $\VB$-groupoids is an affine bundle over the corresponding source fiber of the base groupoid (see \cite[Rem.~3.1.1.(a)]{Bur-Cab-Hoy}).} Hence, $\Upsilon$ integrate to a componentwise multiplicative function on $\bG$ which is skew-symmetric on the $T\G$ components (i.e. there exists $\vartheta \in \Omega^k(\G, \V)$ such that $\Upsilon = \mathrm{Lie}(\c_\vartheta)$)\footnote{We are using \cite[Prop.~A.3]{Bur-Drum} here to ensure that the multiplicative function on $\bG$ is componentwise linear and skew-symmetric.} At last, it follows from comparing \eqref{eq:liner_inf} and \eqref{linear_comp} that $(D,l,\theta)$ satisfies \eqref{eq:inf_comp}. This concludes the proof.
\end{proof}

We end this subsection with an important remark interpreting the Lie functor at the level of differential forms (instead of functions). 
 
\begin{remark}\label{Lie_on_forms}\em
Let $\vartheta \in \Omega^k(\G, \V)$ be a multiplicative form and consider the $k$-form $\tau \in \Omega^k(A, \v)$ such that $\Upsilon = \mathrm{Lie}(\c_{\vartheta}) = \c_\tau$; we interpret $\tau$ as $\mathrm{Lie}(\vartheta)$. Now, consider both $A \to M$ and $\v \to E$ as Lie groupoids with respect to their vector bundle structures and let us apply the Lie functor from $A$ to $\mathrm{Lie}(A) = A$ and from $\v$ to $\mathrm{Lie}(\v)=\v$ (with zero bracket and zero anchor).  The linearity of $\c_\tau$ (with respect to $\bA \to \bM$) is equivalent to $\tau$ being a multiplicative form on $A$ with values in $\v$. Also, as $\mathrm{Lie}(\c_\tau) = \c_\tau$, one gets from comparing \eqref{eq:liner_inf} and \eqref{Flinear_def} that the infinitesimal components of $\vartheta$ and $\tau$ coincide. In particular, applying Theorem \ref{thm:main} to $\tau$, one gets infinitesimal formulas for $D, l, \theta$
\begin{align}\label{eq:linear_corresp}
 D(\eta) = L_{\Delta_{\eta^\uparrow}}(\tau)|_M, \,\, l(\alpha) = i_{\alpha^\uparrow}\tau|_M, \,\, \theta= \tau|_M.
\end{align}
Here, $M \subset A$ as the zero section and one has to notice that the vertical lifts $\eta^\uparrow \in \frakx(\v), \alpha^\uparrow \in \frakx(A)$, for $\eta \in \Gamma_{lin}(E,\v)$ and $\alpha \in \Gamma(A)$, are exactly the right-invariant vector fields.
\end{remark}
 
\subsection{Proof of Theorem \ref{thm:morita}}
We will follow closely the proof of the Morita invariance of the cohomology of 2-term ruth given in \cite{Hoyo-Ort}. 

The main ingredient of the proof is to embed the differential complex $C^{p,q}(\V)$ in the differentiable cochain complex of the big groupoid $\bG$ (see \eqref{big_groupoid}). We shall refer to \cite{Cab-Drum} for the details in what follows. First, one has that 
\begin{equation}\label{form_decomp}
\begin{aligned}
B_p(\times_\G^q T\G) & \cong \times_{B_p\G}^q T(B_p\G)\\
B_p(\times_\G^q T\G \times_\G \V^*) & \cong \times_{B_p\G}^q T(B_p\G) \times_{B_p\G} B_p(\V^*).
\end{aligned}
\end{equation}
The isomorphisms are given by
\begin{align*}
(\underline{U}_1, \dots, \underline{U}_p) & \mapsto (\mathbb{U}^1, \dots, \mathbb{U}^q)\\
((\underline{U}_1, \xi_1), \dots, (\underline{U}_p, \xi_p)) & \mapsto ((\mathbb{U}^1, \dots, \mathbb{U}^q), (\xi_1, \dots, \xi_p)),
\end{align*}
where $\underline{U}_i = (U_i^1, \dots, U_i^q)$, each $U_i^j \in T_{g_i}\G$, for $i=1,\dots, p$, $j = 1, \dots, q$ and $\mathbb{U}^j = (U^j_1, \dots, U^j_p) \in T_{(g_1,\dots, g_p)} B_p\G$. Let us now define $C^{p}_{\rm ext}(\bG)$ as the space of functions $B_p\bG \to \R$ which are multi-linear with respect to the decomposition \eqref{form_decomp} and skew-symmetric on the first $q$-components. From \cite[Prop.~4.7]{Cab-Drum}, we know that $C^{\bullet}_{\rm ext}(\bG)$ is a subcomplex of the differentiable cochain complex of $\bG$ (in fact, there is a chain map $P_{\rm ext}: C^\bullet(\bG) \to C^\bullet_{\rm ext}(\bG)$ which is a projection!).

Now, define $\F: C^{p,q}(\V) \to C^p_{\rm ext}(\bG)$ as follows: for $\vartheta \in C^{p,q}(\V)$, 
$$
\F_\vartheta((\underline{U}_1, \xi_1), \dots, (\underline{U}_p,\xi_p)) = \<\xi_1, \vartheta(\mathbb{U}^1, \dots, \mathbb{U}^q)\>.
$$

\begin{lemma}\label{diff_char}
 A cochain $f \in C^p_{\rm ext}(\bG)$ belongs to the image of $\F$ if and only if it satisfies
 \begin{itemize}
 \item[(i)] $f((\underline{U}_0, 0_{g_0}), (\underline{U}_1, \xi_1),\dots, (\underline{U}_{p-1}, \xi_{p-1})) = 0$;
 \item[(ii)]$f((\underline{U}_0\cdot \underline{U}_1, 0_{g_0}\cdot \xi_1), (\underline{U}_2, \xi_2),\dots, (\underline{U}_p, \xi_p)) = f((\underline{U}_1, \xi_1),\dots, (\underline{U}_p, \xi_p)) $,
 \end{itemize}
for all $((\underline{U}_0, 0_{g_0}),(\underline{U}_1, \xi_1), \dots, (\underline{U}_p, \xi_p)) \in B_{p+1}\bG$.
\end{lemma}

\begin{proof}
 It is straightforward to check $f \in \F(C^{p,q}(\V))$ satisfies both (i) and (ii). Conversely, if $f$ satisfies (i), then we can define a fiber preserving (over $B_p\G$) map $\widehat{f}: B_p(\times_\G^q T\G) \to \pr_1^*\V$ by
 $$
\<\xi_1,  \hat{f}(\underline{U}_1, \dots, \underline{U}_p)\> =  f((\underline{U}_1, \xi_1),\dots, (\underline{U}_p, \xi_p)), \,\,\, \xi_1 \in \V^*_{g_1},
 $$
 where $(\underline{U}_1, \dots, \underline{U}_p) \in B_p(\times_\G^q T\G)$ is in the fiber over $(g_1, \dots, g_p)$ and $(\xi_2, \dots, \xi_p) \in B_{p-1}(\V^*)$ is any element in the fiber over $(g_2, \dots, g_p)$ with $\widetilde{\sour}(\xi_1) = \widetilde{\tar}(\xi_2)$. The multi-linearity and skew-symmetry of $f$ with respect to \eqref{form_decomp} together with (i) implies that $\hat{f}$ is well-defined and that 
 it can be seen as an element $\vartheta \in \Omega^q(B_p\G, \pr_1^*\V)$ such that $f = \F_{\vartheta}$. At last, one can prove that (ii) implies the existence of $\theta \in \Omega^q(B_{p-1}\G, \tar^*E)$ satisfying
 $$
 \widetilde{\sour} \circ \vartheta = \partial_0^*\theta,
 $$
 following exactly the same argument in the proof of Lemma 5.4 of \cite{Gra-Met1}. This concludes the proof.
\end{proof}

One should think of $C^{p,q}(\V)$ inside $C^p_{\rm ext}(\bG)$ in the same way $\VB$-groupoid cochain complex of $\V$ sits inside the complex of linear cochains (see \cite{Gra-Met1}). 

\begin{lemma}
 $(C^{p,q}(\V), \delta)$ is a chain complex and $\F$ is a chain map. 
\end{lemma}

\begin{proof}
 The image of $\F$ is a cochain subcomplex of $C^\bullet_{\rm ext}(\V)$ as one can see by noting that: for $f \in C^{\bullet}_{\rm ext}(\V)$, 
\begin{equation}\label{lem:chain_complex}
 \begin{aligned}
\text{if } f \text{ satisfies (i), then } \delta f \text{ satisfies (i)} & \Leftrightarrow f \text{ satisfies (ii)};\\
f \text{ satisfies (ii)} & \Rightarrow \delta f \text{ satisfies (ii)}.
\end{aligned}
\end{equation}
As $\F$ is a monomorphism, one only has to check that the differential induced by $\F$ on $C^{p,q}(\V)$ coincides with \eqref{differential}. For $\vartheta \in C^{p,q}(\V)$, 
\begin{align*}
\delta \F_\vartheta((\underline{U}_1, \xi_1), \dots, (\underline{U}_{p+1},\xi_{p+1})) & = \<\xi_2, (\partial_0^*\vartheta)(\mathbb{U}^1, \dots, \mathbb{U}^q)\>\\
& \hspace{-120pt} - \<\xi_1 \bullet \xi_2, (\partial_1^*\vartheta)(\mathbb{U}^1, \dots, \mathbb{U}^q)\> + \sum_{i=2}^{p+1} (-1)^i \<\xi_1, (\partial_i^*\vartheta)(\mathbb{U}^1, \dots, \mathbb{U}^q)\>.
\end{align*}
The result now follows from \eqref{dual_mult} and
$$
\partial_1^*\vartheta  = (\partial_1^*\vartheta)\bullet (\partial_0^*\vartheta)^{-1} \bullet (\partial_0^*\vartheta).
$$
\end{proof}

\begin{lemma}
 $\F$ is a quasi-isomorphism.
\end{lemma}

The proof follows exactly the same argument of \cite[Lemma~3.1]{Cab-Drum} with minor modifications (see also \cite[Prop.~4.1]{Hoyo-Ort}). 

\begin{proof}
In the following, we shall identify $C^{p,q}(\V)$ with the image of $\F$ (i.e. the subspace of $C^{p}_{\rm ext}(\bG)$ characterized by Lemma \ref{diff_char}). The result will follow from the following claim: if $f_0 \in C^p_{\rm ext}(\bG)$ satisfies $\delta f_0 \in C^{p,q}(\V)$, then there exists $f_1 \in C^{p-1}_{\rm ext}(\bG)$ such that $f_0 + \delta f_1 \in C^{p,q}(\V)$. 

To prove the claim, first note that, from \eqref{lem:chain_complex}, it suffices to show that if $\delta f_0$ satisfies (i), then there exists $f_1$ such that $f_0+\delta f_1$ satisfies (i).  To do that we shall use a recursion argument. Let us say that $f \in C^p_{\rm ext}(\bG)$ satisfies \textit{(i) up to $l$} if
$$
f((\underline{U}_0, 0_{g_0}), \dots, (\underline{U}_l, 0_{g_l}), (\underline{U}_{l+1}, \xi_{l+1}), \dots,  (\underline{U}_{p-1}, \xi_{p-1})) = 0, 
$$
for every $(\underline{U}_0,\dots, \underline{U}_{p-1}) \in B_p(\times_\G^q T\G)$  and $(\xi_{l+1}, \dots, \xi_{p-1}) \in B_{p-1-l}(\V^*)$ with $\widetilde{\tar}(\xi_{l+1})=0$. It is clear that any $f \in C^p_{\rm ext}(\bG)$ satisfies (i) up to $p$ due to multilinearity. We claim that if $f_0$ is such that $\delta f_0$ satisfies (i) and $f_0$ satisfies (i) up to $l$, then there exists $f_1 \in C^{p-1}_{\rm ext}(\V)$ such that $f_0+ \delta f_1$ satisfies (i) up to $l-1$. Indeed, define $f_1$ as follows:
$$
f_1((\underline{U}_1, \xi_1), \dots, (\underline{U}_{p-1},\xi_{p-1})) = - f_0((\underline{U},\xi), (\underline{U}_1, \xi_1), \dots, (\underline{U}_{p-1}, \xi_{p-1})),
$$
where 
$$
\underline{U}=\underline{U}_{p-1}^{-1} \bullet \dots \bullet \underline{U}_1^{-1},  \,\,\, \xi= h(g_{p-1}^{-1}\dots g_1^{-1}, \widetilde{\tar}(\xi_1)),
$$
$h: \sour^*C^* \to \V^*$ is a horizontal splitting. It is a straightforward computation to check that $f_1$ is indeed multi-linear and skew-symmetric. Also, 
\begin{align*}
(f_0+\delta f_1)((\underline{U}_0, 0_{g_0}), \dots, (\underline{U}_{l-1}, 0_{g_{l-1}}), (\underline{U}_l, \xi_l), \dots, (\underline{U}_{p-1}, \xi_{p-1})) & = \\
& \hspace{-260pt} = (\delta f_0)((\underline{U}_{p-1}^{-1} \bullet \dots \bullet \underline{U}_0^{-1}, 0_{g_{p-1}^{-1}}\dots 0_{g_0^{-1}}), (\underline{U}_0, 0_{g_0}), \dots, (\underline{U}_{p-1}, \xi_{p-1}))\\
& \hspace{-260pt} = 0.
\end{align*}
The proof now follows from the fact that $f$ satisfies (i) up to 0 if and only if it satisfies (i). 
\end{proof}

We are now able to complete the proof of Theorem \ref{thm:morita}.

\begin{proof}[Proof of Theorem \ref{thm:morita}]
 It follows as a Corollary of \cite[Thm.~3.5]{Hoyo-Ort} that the map 
 $$
 F: \underbrace{(\times_{\G'}^q T\G') \times_\G \phi^*(\V^*)}_{\bG'} \to (\times_\G^q T\G) \times_\G \V^*
 $$ induced by the differential of $\phi: \G' \to \G$ and the natural pull-back map $\phi^*(\V^*) \to \V^*$ is a Morita map. As such it defines an isomorphism between the differentiable cochain cohomology $H^\bullet(\bG) \stackrel{\sim}\to H^\bullet(\bG')$. Now, it is straightforward to see that $F$ preserves the $C^\bullet_{\rm ext}$ complexes and the existence of a chain projection $P_{\rm ext}: C^\bullet \to C^\bullet_{\rm ext}$ (see \cite{Cab-Drum}) implies that $F$ induces an isomorphism between the $\rm ext$-chain cohomologies. Now, the proof follows from the following diagram
 $$
 \begin{CD}
  C^{p,q}(\V) @> \phi^* >> C^{p,q}(\phi^*\V) \\
  @V\F VV                       @VV \F V\\
  C^p_{\rm ext }(\bG) @>> F^* > C^p_{\rm ext}(\bG')
  \end{CD}
 $$
 and the fact that $\F$ is a quasi-isomorphism.
\end{proof}

\section{Multiplicative distributions}
In this section, we study multiplicative distribution on Lie groupoids. These are subbundles $\H\subset T\G$ such that $\H \toto \H_M$ is a subgroupoid of $T\G \toto TM$, where $\H_M \subset TM$. The Lie theory of such structures were studied in \cite{CSS} in the case $\H_M = TM$, which we refer here as the \textit{wide case}. We aim here to extend their results to the general case.

\subsection{Some properties}
\noindent
For a multiplicative distribution $\H \toto \H_M$, note that $T\sour(\H) = \sour^*\H_M$ and the exact sequence \eqref{core_ses} induces
\begin{equation}\label{distr_ses}
 0 \to \tar^*K  \stackrel{r}\to \H\stackrel{Ts} \to s^*\H_M \to 0,
\end{equation}
where $K=\H|_M \cap A$ and $r(g,k) = k \bullet 0_g$ is the right-multiplication.  In general, we say that an arbitrary distribution $\H \subset T\G$ is \textit{right-invariant} if it fits in the short exact sequence \eqref{distr_ses}. Note that, for right-invariant distributions, $T\sour: \H \to \H_M$ is automatically a surjective submersion.  We shall refer to the pair $(\H_M, K)$ as the \textit{profile} of the right-invariant distribution $\H$.

\begin{example}\em
For a regular groupoid (i.e. the orbits of $\G$ have constant rank), the isotropy distribution $\H= \ker(T\sour) \cap \ker(T\tar)$ is multiplicative. Here,
$$
\H_M = 0, \,\, K = \ker(\rho).
$$
In the case $\G = \G(P)$, the gauge groupoid associated to a principal bundle $P \to M$, it can be proved that connections on $P$ correspond to multiplicative complements to the isotropy distribution (see \cite[Prop.~3.4]{Bur-Drum2}).
\end{example}

\begin{example}\em
For a multiplicative distribution on a Lie group, $\H_M = 0$ and $K$ is an ideal on the Lie algebra. The distribution is the bi-invariant distribution corresponding to $K$ \cite{Ort1}.
\end{example}

\begin{example}\em
A Cartan connection on $\G$ is a multiplicative distribution with profile
$$
\H_M = TM, \,\, K = 0.
$$
The existence of a Cartan connection on $\G$ imposes strong restrictions on the Lie groupoid. Indeed, assuming $\G$ has source 1-connected fibers and $M$ is compact, simply connected, the existence of a Cartan connection is equivalent to $\G$ being the action groupoid corresponding to a Lie group action on $M$ (see \cite[Rem.~2.14]{Arias-Crai1}).
\end{example}

\begin{example}\em
 Multiplicative distributions on a vector bundle $E \to M$ (seen as a Lie groupoid with multiplication given by fiberwise sum) are called \textit{linear distributions}. An important example is given by the horizontal distribution corresponding to a connection $\nabla: \mathfrak{X}(M) \times \Gamma(E) \to \Gamma(E)$. In fact, any linear distribution on $E$ with profile $(TM, 0)$ is the horizontal distribution for a connection on $E$.
\end{example}

For a right-invariant distribution, let us consider pointwise splittings of \eqref{distr_ses} as follows: 
$$
\F_\H(T\G) = \{ (g,b) \in \F(T\G) \,\, | \,\, b(X) \subset \H, \,\,\, \forall \, X \in (\H_M)_{\sour(g)}\}, \,\,\, J^1_\H\G= J^1\G \cap \F_\H(T\G).
$$
The next Proposition gives criteria to a right-invariant distribution to be multiplicative.

\begin{proposition}
 A right-invariant distribution $\H$ with profile $(\H_M, K)$ is multiplicative if and only if 
\begin{itemize}
\item[(i)]$\rho(K) \subset \H_M$;
\item[(ii)]$\F_\H(T\G) \toto M$ is a Lie subcategory of $\F(T\G) \toto M$;
\item[(iii)] The fat  representation of $\F(T\G)$ on $\rho: A \to TM$ restricts to a representation of $\F_\H(T\G)$ on $\rho: K \to \H_M$.
\end{itemize}
\end{proposition}

\begin{proof}
 It is straightforward to check that if $\H$ is multiplicative, then (i), (ii) and (iii) holds. Conversely, let us assume that (i), (ii) and (iii) hold. First note that $\H_M = \H|_M \cap TM$, since the units are in $\F_\H(T\G)$.  Also, $T\tar(\H) \subset \H_M$ because $\rho(K)\subset \H_M$ and $\Psi_b$ preserves $\H_M$, for $b \in \F_\H(T\G)$. It will be a surjective submersion onto $\H_M$ once we prove that the inversion preserves $\H$  (since $T\tar = T\sour \circ \iota$). Let $U_1, U_2 \in \H$ be composable vectors and write
 $
 U_1 = b_1(X_1) + k_1 \bullet 0_{g_1}, \,\,\, U_2 = b_2(X_2) + k_2\bullet 0_{g_2},
 $
 for $b_1, b_2 \in \F_\H(T\G)$, $X_1, X_2 \in \H_M$ and $k_1, k_2 \in K$. One can check that
 $$
 U_1\bullet U_2 = (b_1 \cdot b_2)(X_2) + (\Psi_{b_1}(k_2) + k_1) \bullet 0_{g_1g_2} \in \H.
 $$
 Let us now prove that $U \in \H \Rightarrow U^{-1} \in \H$. Write $U=b(X)+k\bullet 0_g$, for $b \in \F_H(T\G)$ and $k \in K$. Note that 
 $$
 k^{-1} = \rho(k) - k \in \H_M \oplus K \subset \H|_M \Rightarrow (k\bullet 0_g)^{-1}=0_{g^{-1}}\bullet k^{-1} \in \H.
 $$
Choose $\hat{b}:T_{\tar(g)}M \to T_{g^{-1}}\G$ in $\F_\H(T\G)$ and write $b(X)^{-1} = \hat{b}(\Psi_b(X)) + c \bullet 0_{g^{-1}}, \,\, \text{ for } c \in A$. Now,
$$
X=b(X)^{-1} \bullet b(X) = (\hat{b}\cdot b)(X) + c \Rightarrow c \in \H|_M \cap A = K \Rightarrow b(X)^{-1} \in \H.
$$
Hence, $U^{-1} =b(X)^{-1} + (k \bullet 0_{g})^{-1} \in \H$ and this proves that $\H \toto \H_M$ is a Lie subgroupoid.
\end{proof}

From now on, let us fix a multiplicative distribution $\H \subset T\G$ with profile $(\H_M, K)$, consider the $\VB$-groupoid $\V = T\G/\H \toto TM/\H_M$, with core $A/K$. The core anchor of $\V$ is the map induced by $\rho: A \to TM$ on the quotient bundles, $\overline{\rho}: A/K \to TM/\H_M$. Any $(g,b) \in \F_{\H}(T\G)$ defines a linear map from $(TM/\H_M)_{\sour(g)} \to (T\G/\H)_{g}$ which we denote by $[(g,b)]$. Also,
$$
\F_{\rm inv}(\V) = \{[(g,b)] \,\,| \,\, (g,b) \in J^1_{\H}\G\}.
$$
The fat representation of $\F_{\rm inv}(\V)$ on $\overline{\rho}: A/K \to TM/\H_M$ is given by 
$\Psi_{[(g,b)]} = [\Psi_{(g,b)}]$, where $[\Psi_{(g,b)}]$ is the quotient chain map induced by the adjoint representation of $J^1\G$ on $\rho:A \to TM$.

Infinitesimally, we have a similar picture. Let $\frakh \subset TA$ and $\v = TA/\frakh$ be the $\VB$-algebroids of $\H$ and $T\G/\H$, respectively. Define 
$$
J^1_\frakh A = \{ \eta \in J^1A\,\,|\,\, \eta(\H_M) \subset \frakh\}, \,\,\,\Hom_\frakh(TM,A) = J^1_\frakh A \cap \Hom(TM,A).
$$
The short exact sequence \eqref{linear_ses} for $TA \to TM$ induces
$$
0 \to \Hom_\frakh(TM,A) \to J^1_\frakh A \to A \to 0.
$$
The surjectiveness of $J^1_\frakh A \to A$ follows from the existence of adapted connections \footnote{A connection $\nabla: \mathfrak{X}(M) \times \Gamma(A) \to \Gamma(A)$ is \textit{adapted} if the linear vector fields on $A$ corresponding to the derivations $\nabla_X$, $X \in \H_M$, belong to the distribution $\frakh$.} (see \cite[Prop.~5.5]{Dru-Jotz-Ort}). Any section $\eta \in \Gamma(J^1_\frakh A)$ induces a linear section $[\eta]: TM/\H_M \to \v$ and the map $\eta \mapsto [\eta]$ defines a surjection  $\Gamma(J^1_\frakh A) \to \Gamma_{lin}(TM/\H_M, \v)$. 

\begin{example}\em[Wide case]
In the case $\H_M = TM$, the $\VB$-groupoid $\V$ has no side bundle, i.e. $ \V=\tar^*(A/K)$. So, in this case, $\F_{\rm inv}(\V) \cong \G$ and the fat representation on $A/K$ coincides with the $\G$ representation defined as follows:
for $g \in \G$ and $a \in A_{\sour(g)}$, choose any $U \in \H_g$ such that $T\sour(U) = \rho(a)$ and define
\begin{equation*}
g \cdot [a] = [U \bullet a \bullet 0_{g^{-1}}],
\end{equation*}
where $[\cdot]$ is the class mod $K$. It is straightforward to check that $\cdot$ does not depend on the choices made. Similarly, $\F(\v) \cong A$ and the fat representation of $A$ on $A/K$ is given as follows:
\begin{equation}\label{A_connection}
\nabla_\alpha [\beta] = \pi_A([\alpha, \beta] + D^{\rm clas}_{\rho(\beta)}(\eta)),
\end{equation}
where $\alpha, \beta \in \Gamma(A)$, $\pi_A: A \to A/K$ is the quotient projection and $\eta \in \Gamma(J^1_\frakh A)$ is any section such that $\pr(\eta)=\alpha$. Again, the formula is well-defined and it is a consequence of \eqref{jet_representation}.
\end{example}


\subsection{IM distributions}
We shall now focus on the infinitesimal picture. It is well-known (see \cite[Thm.~5.7]{Dru-Jotz-Ort}) that there is a 1-1 correspondence between linear distributions $\frakh \subset TA$ with profile $(\H_M, K)$ and operators $\bbD: \Gamma(A) \to \Gamma(\H_M^* \otimes (A/K))$ satisfying a Leibniz identity
\begin{equation}\label{bbd:leibniz}
\bbD_X(f\alpha) = f\bbD_X(\alpha) + (\Lie_Xf) \pi_A(\alpha).
\end{equation}
The operator is obtained from $\frakh$ as follows:
\begin{equation}\label{bbD2}
\bbD_X(\alpha) =  \pi_A([\widehat{X}, \alpha^\uparrow](0_x)), \,\, X \in (\H_M)_x,
\end{equation}
where $\widehat{X} \in \Gamma(A, \frakh)$ is a projectable vector field (with respect to $TA \to TM$) satisfying $\widehat{X}(0_x) = X$.
Reciprocally, one can reconstruct $\frakh$ from $\bbD$ by using the splitting $\sigma: A \to J^1A$ of \eqref{linear_ses} associated to any  connection $\nabla$ on $A$ \textit{adapted to $\bbD$} (i.e. such that $\pi_A(\nabla_X \alpha) = \bbD_X(\alpha)$, for all $X \in \H_M$). In this case, 
\begin{equation}\label{h_from_D}
\frakh_\bbD = \{\sigma(a)(X) +_A (0_a +_{TM} \overline{k}) \,\, | \,\, (a,k,X) \in A \times_M K \times_M \H_M\}.
\end{equation}
\begin{remark}\em
 In the wide case, formula \eqref{bbD2} implies that $\bbD$ is the Spencer operator associated to $\frakh$ obtained in \cite{CSS}.
\end{remark}

Let us now define
$$
\mathcal{J}_\bbD = \{ \eta \in \Gamma(J^1A) \,\, | \,\, \pi_A(D^{\rm clas}_X(\eta)) = \bbD_X(\pr(\eta)), \,\, \forall \, X \in \H_M\}.
$$

\begin{proposition}
 For a linear distribution, one has that 
 $$
 \mathcal{J}_\bbD = \Gamma(J^1_\frakh A).
 $$
 Moreover, for $\eta \in \mathcal{J}_\bbD$, the following facts hold for the fat connection $\nabla_\eta$ on $\rho: A \to TM$:
 \begin{itemize}
 \item[(a)] $\nabla_\eta \text{ preserves } \H_M  \Longleftrightarrow  \overline{\rho}(\bbD_{X}(\alpha)) = -\pi_M([\rho(\alpha), X ])$, $\forall \, X \in \Gamma(\H_M)$,
 \item[(b)] $\nabla_\eta \text{ preserves } K  \Longleftrightarrow  \bbD_{\rho(k)}(\alpha) = -\pi_A([\alpha, k]), \,\,\, \forall \, k \in \Gamma(K).$
 \end{itemize}
where $\alpha = \pr(\eta)$.  
\end{proposition}

\begin{proof}
 Choose any connection $\nabla$ on $A$ adapted to $\bbD$ and let $\sigma: A \to J^1A$ be the corresponding splitting of \eqref{linear_ses}. Decompose $\eta = \sigma(\alpha) + B\Phi$, where $\alpha= \pr(\eta)$ and $\Phi: TM \to A$ is a vector bundle morphism. Now, from \eqref{h_from_D}, $\eta(\H_M) \subset \frakh_\bbD$ if, and only if, $\Phi(X) \in K$, for all $X \in \H_M$. Now, using that $D^{\rm clas}_X(\eta) = \nabla_X \alpha - \Phi(X)$, one has that
 $$
\pi_A(\Phi(X)) = 0 \Leftrightarrow \pi_A( \nabla_X \alpha - D^{\rm clas}_X(\eta)) = 0 \Leftrightarrow \pi_A(D^{\rm clas}_X(\eta)) = \bbD_X(\alpha).
 $$
 The rest of the Proposition follows from the explicit formulas \eqref{jet_representation} of the fat representation of $J^1A$ on $\rho: A \to TM$.
\end{proof}

In case the fat connection $\nabla_\eta$ preserves the profile $(\H_M, K)$ of a linear distribution $\frakh$, for $\eta \in J^1 A$, we define $\overline{\nabla}_\eta$ as the connection on the quotient complex $\overline{\rho}: A/K \to TM/\H_M$. 

\begin{definition}
 An IM distribution on a Lie algebroid $A$ is a triple $(\H_M, K, \mathbb{D})$, where $\H_M \subset TM$, $K \subset A$ are subbundles with $\rho(K) \subset \H_M$, $
\mathbb{D}: \Gamma(A) \to \Gamma(\H_M^* \otimes (A/K))
$
is a $\mathbb{R}$-linear operator satisfying the Leibniz condition \eqref{bbd:leibniz}
and the IM equations:
\begin{align}
\label{IM_dist1}   \overline{\rho}(\mathbb{D}_X(\alpha)) & = - \pi_M([\rho(\alpha),X]),\\
\label{IM_dist2}   \mathbb{D}_{\rho(k)}(\alpha) & = -\pi_A([\alpha,k]),\\
\label{IM_dist3}   \mathbb{D}_X([\pr(\eta_1), \pr(\eta_2)]) & = \overline{\nabla}_{\eta_1} \mathbb{D}_X(\pr(\eta_2)) - \pi_A(D^{\rm clas}_{[\rho(\pr(\eta_1)), X]}(\eta_2))\\
\nonumber & \hspace{-30pt} - \overline{\nabla}_{\eta_2} \mathbb{D}_X(\pr(\eta_1)) + \pi_A(D^{\rm clas}_{[\rho(\pr(\eta_2)), X]}(\eta_1)),
\end{align}
for  $\alpha \in \Gamma(A), k \in \Gamma(K), \, X \in \Gamma(\H_M), \,  \eta_1, \eta_2 \in \mathcal{J}_\bbD$.
\end{definition}

In the following remarks, we show how IM distributions relate to the Spencer operators obtained in \cite{CSS} and to the infinitesimal data obtained in \cite{Dru-Jotz-Ort}.

\begin{remark}\em
In the case $\H_M = TM$, one has that $\overline{\rho} \equiv 0$ and $\pi_M \equiv 0$, so \eqref{IM_dist1} is trivially satisfied. Also, for $\eta \in \mathcal{J}_\bbD$, the quotient connection $\overline{\nabla}_\eta$ only depends on $\pr(\eta)=\alpha$. Indeed,
$$
\overline{\nabla}_\eta [\beta] = \pi_A([\alpha, \beta]) + \pi_A(D^{\rm clas}_{\rho(\beta)}(\eta)) = \pi_A([\alpha, \beta]) + \bbD_{\rho(\beta)}(\alpha)
$$
(this is exactly the $A$-connection \eqref{A_connection}). Hence, \eqref{IM_dist3} can be rewritten as:
\begin{align*}
 \mathbb{D}_X([\alpha_1, \alpha_2]) = \nabla_{\alpha_1} \mathbb{D}_X(\alpha_2) - \bbD_{[\rho(\alpha_1), X]}(\alpha_2)
- \nabla_{\alpha_2} \mathbb{D}_X(\alpha_1) + \bbD_{[\rho(\alpha_2), X]}(\alpha_1),
\end{align*}
So, in the wide case, an IM distribution $(TM, K, \bbD)$ is the same as a Spencer operator on $A$ relative to $K$ (see \cite[Dfn.~2.16]{CSS}).
\end{remark}
%
%
\begin{remark}\em  
By choosing a connection $\nabla^0$ adapted to $\bbD$, one can show that \eqref{IM_dist3} is equivalent to equation (5.16) in \cite[Thm.~5.17]{Dru-Jotz-Ort} (a result giving conditions to linear distribution on $A$ to be a $\VB$-subalgebroid of $TA \to TM$). Indeed, any linear section $\eta \in \mathcal{J}_\bbD$ can be written as
 $
\eta = j^1\alpha - (\B \nabla^0_{\cdot} \alpha + \B\Phi), 
 $
 where $\Phi: TM \to A$ satisfies $\Phi(\H_M) \subset K$. Using this decomposition, one can check that \eqref{IM_dist3} can be re-written as
 \begin{align*}
 \bbD_X([\alpha_1, \alpha_2]) & = \underbrace{\widehat{\nabla}^{\rm bas}_{\alpha_1}\bbD_X(\alpha_2) - \widehat{\nabla}^{\rm bas}_{\alpha_2}\bbD_X(\alpha_1) + \pi_A(\nabla^0_{[\rho(\alpha_2), X]}) \alpha_1 - \nabla^0_{[\rho(\alpha_1), X]} \alpha_2)}_{(\ast)}\\
 & \hspace{-50pt} + \underbrace{ \pi_A(\Phi_1(\nabla_{\alpha_2}^{\rm bas} X) - \Phi_2(\nabla_{\alpha_1}^{\rm bas} X) - \Phi_1\circ \rho \circ \Phi_2(X) + \Phi_2 \circ \rho \circ \Phi_1(X))}_{(\ast \ast)}, 
 \end{align*}
 where $\nabla^{\rm bas}$ is the $A$-connection of the adjoint representation up to homotopy associated to $\nabla^0$ and $\widehat{\nabla}^{\rm bas}_{\alpha}$ is the quotient $A$-connection on $\overline{\rho}:A/K \to TM/\H_M$. Now, $(\ast)$ is exactly the expression appearing in \cite{Dru-Jotz-Ort} and it is straightforward to check that $(\ast \ast) = 0$ using that $\Phi_i(\H_M) \subset K$ and that $\nabla^{\rm  bas}$ preserves $\H_M$.
\end{remark}

The following Proposition gives alternative characterizations of IM distributions.

\begin{proposition}\label{prop:IM_equiv}
 Let $K \subset A$, $\H_M \subset TM$ be subbundles and $\bbD: \Gamma(A) \to \Gamma(\H_M^*\otimes A/K)$ be an $\R$-linear operator satisfying the Leibniz equation \eqref{bbd:leibniz}. Consider the linear distribution \eqref{h_from_D} $\frakh \subset TA$ corresponding to $\bbD$. The following are equivalent:
 \begin{itemize}
  \item[(a)] $(\H_M, K, \bbD)$ is an IM distribution;
  \item[(b)] $\frakh \subset TA$ is a $\VB$-subalgebroid;
  \item[(c)] $J^1_{\frakh} A \subset J^1A$ is a Lie subalgebroid, $\rho(K) \subset \H_M$ and the fat representation of $J^1A$ on $\rho: A \to TM$ restricts to a representation of $J^1_\frakh A$ on $\rho: K \to \H_M$; 
  \item[(d)] The quotient double vector bundle $\v:=TA/\frakh \to TM/\H_M$ with core $A/K$ is a $\VB$-algebroid and $\pi_A: A \to A/K$, $\pi_M: TM \to TM/\H_M$ and the operator $D: \Gamma_{lin}(TM/\H_M, \v) \to \Omega^1(M, A/K)$ given by
  $$
  D([\eta]) = \pi_A(D^{\rm clas}(\eta)), \, \text{ for } \eta \in J^1_\frakh A
  $$
  define an IM 1-form on $A$ with values in $\v$.
 \end{itemize}
\end{proposition}

\begin{proof}
\noindent
\paragraph{$(a) \Leftrightarrow (b)$} It is the content of Theorem 5.17 of \cite{Dru-Jotz-Ort}.
\medskip

\paragraph{$(b) \Leftrightarrow (c)$} It follows from the description of the Lie algebroid structure on $TA \to TM$ using linear and core sections. More concretely, $\Gamma(\H_M,\frakh)$ is generated by $\eta \in \Gamma(J^1_\frakh A)$ and $\B k$, for $k \in K$, as a $C^\infty(\H_M)$-module. Now, from Proposition \ref{prop:derivation},
$$
\rho_{TA}(\eta) = W_{\nabla_\eta} \in T\H_M \Leftrightarrow \nabla_\eta \text{ preserves } \H_M,
$$
where $\nabla_\eta$ is the fat connection on $TM$. Similarly, $\rho_{TA}(\B k) = \rho(k)^\uparrow \in T\H_M \Leftrightarrow \rho(k) \in \H_M$. As for the Lie bracket,  $\Gamma(\H_M,\frakh)$ will be involutive if and only if $J^1_\frakh A \subset J^1A$ is a Lie subalgebroid and $\nabla_\eta$ preserves $K$, for every $\eta \in J^1_\frakh A$.
\medskip

\paragraph{$(b) \Rightarrow (d)$} The fact that $\frakh \subset TA$ is a $\VB$-subalgebroid implies that $\v=TA/\frakh \to TM/\H_M$ is a $\VB$-algebroid and the quotient map $TA \to \v$ is a $\VB$-algebroid morphism. The $\VB$-algebroid structure on $\v$ is determined by the following equations: for $\eta, \eta_1, \eta_2 \in \Gamma(J^1_\frakh A)$, $\beta \in \Gamma(A)$, $X \in \mathfrak{X}(M)$,
\begin{align}
\label{IM1_dist}[[\eta_1], [\eta_2]] & = [\,[\eta_1, \eta_2]\,] \,\,\,\,\, \textit{(linear-linear bracket)}\\
\label{IM2_dist}\nabla_{[\eta]}\pi_A(\beta) & = \pi_A(\nabla_\eta \beta) \,\,\,\textit{(linear-core bracket)}\\ 
\label{IM4_dist}\nabla_{[\eta]} \pi_M(X) & = \pi_M(\nabla_\eta X))\,\,\,\,\,\, \textit{(anchor on linear sections)}\\
\label{IM5_dist}\partial(\pi_A(\beta)) &= \pi_M(\rho(\beta))\,\,\,\,\, \textit{(anchor on core sections)}
\end{align}
where $\nabla_{[\eta]}$ is the fat representation on the core anchor complex $\partial: A/K \to TM/\H_M$. It is straightforward to check that \eqref{IM2_dist}, \eqref{IM4_dist} and \eqref{IM5_dist} are exactly the IM-equations (IM2), (IM4) and (IM5), respectively, for $(D,\pi_A,\pi_M)$. Also, \eqref{IM1_dist} and the fact that $(D^{\rm clas}, \mathrm{id}_A, \mathrm{id}_{TM})$ is a IM 1-form on $A$ with values in $TA$ imply the remaining equation (IM1) for $(D,\pi_A, \pi_A)$. Indeed,
\begin{align*}
D_X([[\eta_1],[\eta_2]]) & =  \pi_A(D_X^{\rm clas}([\eta_1,\eta_2])) \\
 & \hspace{-50pt} = \pi_A\left(\nabla_{\eta_1} D_X^{\rm clas}(\eta_2)  - D^{\rm clas}_{[\rho(\alpha_2), X]}(\eta_1) - \nabla_{\eta_2} D_X^{\rm clas}(\eta_1)  + D^{\rm clas}_{[\rho(\alpha_1), X]}(\eta_2)\right)\\
 & \hspace{-50pt} = \nabla_{[\eta_1]} D_X([\eta_2]) - D_{[\rho(\alpha_2), X]}([\eta_1]) - \nabla_{[\eta_2]} D_X([\eta_1]) + D_{[\rho(\alpha_1), X]}([\eta_2]).
 &  
\end{align*}
\paragraph{$(d) \Rightarrow (c)$} The IM-equations (IM2), (IM4) and (IM5) for $(D, \pi_A, \pi_M)$  are exactly \eqref{IM2_dist}, \eqref{IM4_dist} and \eqref{IM5_dist}, respectively. It is straightforward to check that: \eqref{IM5_dist} implies that $\rho(K) \subset \H_M$ and \eqref{IM2_dist} together with \eqref{IM4_dist} implies that the fat connection $\nabla_\eta$ preserves $\rho: K \to \H_M$ if $\eta \in J^1_\frakh A$. The remaining equation (IM1) can be written as
$$
D([[\eta_1], [\eta_2]]) = \pi_A(D^{\rm clas}([\eta_1, \eta_2])),
$$
for $\eta_1, \eta_2 \in \Gamma(J^1_\frakh A)$. Let $\eta \in \Gamma(J^1_\frakh A)$ be any section such that $[\eta] = [[\eta_1], [\eta_2]]$. As $\pr(\eta) = \pr([\eta_1, \eta_2])$, there exists $\Phi: TM \to A$ such that $\eta = [\eta_1, \eta_2] + \B \Phi$. Now, as
$$
\pi_A(D^{\rm clas}(\eta)) = D([[\eta_1], [\eta_2]]) = \pi_A(D^{\rm clas}[\eta_1, \eta_2]),
$$
one has that $\Phi(TM) \subset K$. This implies that $[\eta_1,\eta_2] \in \Gamma(J^1_\frakh A)$ as we wanted to prove.
\end{proof}

We are now able to state our main result regarding multiplicative distributions. It generalizes \cite[Thm.~2]{CSS}.

\begin{theorem}\label{thm:IM_dist}
 Let $\G \toto M$ be a source 1-connected groupoid. There is a 1-1 correspondence between multiplicative distributions $\H \subset T\G$ with profile $(\H_M, K)$ and IM distributions on $A$, $\bbD: \Gamma(A) \to \Gamma(\H_M^* \otimes (A/K))$. For $X \in (\H_M)_x$, 
 \begin{equation}\label{IM_dist_integration}
 \bbD_X(\alpha) = \pi_A([\widetilde{X}, \overrightarrow{\alpha}]),
 \end{equation}
 where $\widetilde{X} \in \Gamma(\H)$ is any section with $\widetilde{X}(x) = X$.
\end{theorem}

\begin{proof}
 The 1-1 correspondence between multiplicative distributions and IM distributions is a straightforward consequence of the Lie theory of $\VB$-groupoids. Indeed, given a multiplicative distribution $\H \subset T\G$, its $\VB$-algebroid $\frakh \subset TA$ correspond to an IM distribution via \eqref{bbD2} and Proposition \ref{prop:IM_equiv}. Reciprocally, given an IM distribution, the corresponding $\VB$-subalgebroid $\frakh \subset TA$ \eqref{h_from_D} can be integrated to a $\VB$-subgroupoid $\H \subset T\G$ using \cite[Cor.~4.3.7]{Bur-Cab-Hoy}. 
 
 So, what remains to be shown is \eqref{IM_dist_integration}. For this, consider the quotient projection $\vartheta: T\G \to T\G/\H$. It is a multiplicative 1-form with values in the $\VB$-groupoid $\V:= T\G/\H$. Hence, from Theorem \ref{thm:main}, there exists an IM 1-form $(D, l, \theta)$ on $A$ with values in $\v:=TA/\frakh$, where $\frakh \subset TA$ is the $\VB$-algebroid of $\H$ and 
 $$
 D_X(\eta) = \Delta_{\overrightarrow{\eta}}(\vartheta(\widetilde{X})) - \vartheta([\overrightarrow{\alpha}, \widetilde{X}]),
 $$
where $\eta \in \Gamma_{lin}(TM/\H_M, \v)$, $\alpha = \pr(\eta)$, $X \in T_xM$ and $\widetilde{X} \in \frakx(\G)$ is any vector field with $\widetilde{X}(x)$. If $X \in \H_M$, one can always choose $\widetilde{X} \in \Gamma(\H)$ so that $\vartheta(\widetilde{X})=0$. In this case, 
$$
D_X(\eta) = \vartheta([\widetilde{X}, \overrightarrow{\alpha}](x)) = \pi_A([\widetilde{X}, \overrightarrow{\alpha}](x)), \,\,\,\, \forall \, X \in \H_M,
$$
where we have used that $\vartheta|_A = l$ and $l=\pi_A: A \to A/K$, the quotient projection. This result can now be applied to the linear distribution $\frakh \subset TA$ to obtain the following fact regarding the IM distribution \eqref{bbD2}:
$
D^{\tau}_X(\eta) = \bbD_X(\alpha).
$
where $\tau: TA \to TA/\frakh$ is the quotient projection and $D^{\tau}$ is the IM 1-form associated to $\tau$. The result now follows from the coincidence between the infinitesimal components of $\vartheta$ and $\tau=\mathrm{Lie}(\vartheta)$ (see Remark \eqref{Lie_on_forms}).
\end{proof}

\begin{remark}\em
The problem of involutivity of $\H \subset T\G$ is studied in \cite{CSS,Jotz-Ortiz}.
 \end{remark}

\appendix

\section{Linear vector fields}\label{app:linear}
In this appendix, we recall some well-known facts regarding linear vector fields on vector bundles and their correspondence to derivations. The main goal is to set some notation used in the text and to establish a result giving criteria to a linear vector field to be tangent to a subbundle in terms of the corresponding derivation. The reference for this part is \cite{Kos-Mac}, \cite[Section~3.4]{Mckz2}. 

Let $p:E \to M$ be a vector bundle.  A vector field $W \in \frakx(E)$ is said to be \textit{linear} if it flows by linear isomorphisms of $E$, $\Fl_{W}^\epsilon: E \to E$. The flow of $W$ covers the flow of a vector field on $M$ which we denote by $\sharp(W)$ and called it the symbol of $E$.  A linear vector field $W$ corresponds to a derivation \footnote{A derivation on $E$ is an $\R$-linear map $\Delta: \Gamma(E^*) \to \Gamma(E^*)$ satisfying $\Delta(f\psi) = f\Delta(\psi) + (\Lie_{\sharp(\Delta)}f) \,\psi$, for $f \in C^\infty(M)$. The vector field $\sharp(\Delta)$ is called the symbol of $\Delta$.} $\Delta: \Gamma(E^*) \to \Gamma(E^*)$ on $E$. The correspondence between derivations and linear vector fields works as follows: a section $\psi \in \Gamma(E^*)$ gives rise to a function $\ell_\psi \in C^{\infty}(E)$ and
\begin{align} 
\label{eq:der1} W(\ell_\psi) & = \ell_{\Delta \psi}\\
\label{eq:der2} W(f \circ p_*) & = (\Lie_{\sharp(\Delta)} f) \circ p_*,
\end{align}
where $p_*: E^* \to M$ is the projection on the dual vector bundle. In particular, the symbols coincide: $\sharp(W) = \sharp(\Delta)$. 

Given a derivation $\Delta: \Gamma(E^*) \to \Gamma(E^*)$,  its \textit{adjoint} is the derivation $\Delta^{\top}: \Gamma(E) \to \Gamma(E)$ on $E$ defined as follows:
\begin{equation}\label{dual_der}
\<\psi, \Delta^\top(u)\> = \Lie_{\sharp(\Delta)} \<\psi, u\> -\<\Delta(\psi), u\>, \,\,\, \psi \in \Gamma(E^*), \,u \in \Gamma(E).
\end{equation}
There is another important characterization of the adjoint derivation $\Delta^\top$ (see Theorem 3.4.5 in \cite{Mckz2}): for $u \in \Gamma(E)$
\begin{equation}\label{der_adjoint}
(\Delta^\top(u))^\uparrow = [W_\Delta, u^\uparrow],
\end{equation}
where $(\cdot)^\uparrow$ is the vertical lift of a section.

Taking adjoints gives a 1-1 correspondence $W \in \frakx(E) \mapsto W^\top \in \frakx(E^*)$ between linear vector fields on $E$ and on $E^*$. It will be important to give a more concrete description of this correspondence. Let $\brac{\cdot, \cdot}: TE \times_{TM} T(E^*) \to \R$ be the derivative of the natural bracket $\<\cdot, \cdot\>: E\times_M E^* \to \R$. The equation
\begin{equation}\label{eq:dual_linear}
\brac{W^\top(\psi), W(u)} = 0 , \,\,\, \forall \, (\psi, u) \in E^*\times_M E,
\end{equation}
completely determine the duality relation (see Proposition 3.4.7 in \cite{Mckz2}). Also, the flow of $W^\top$ is given by
\begin{equation}\label{eq:adj_flow}
 \Fl_{W^\top}^\epsilon = (\Fl_W^{-\epsilon})^*.
\end{equation}

\begin{example}\em\label{ex:tan_lift}
 A vector field $X \in \frakx(M)$ on $M$ gives rise to a linear vector field $X^T \in \frakx(TM)$ on $TM \to M$ called the tangent lift of $X$. Its flow is given by differentiation of the flow of $X$, $T \Fl_X^\epsilon: TM \to TM$. The corresponding derivation on $T^*M$ is the Lie derivative along $X$,
 $$
 \Lie_X: \omega \mapsto \Lie_X \omega.
 $$
 The adjoint derivation  is $[X, \cdot]: \Gamma(TM) \to \Gamma(TM)$
 \end{example}

\begin{remark}\label{der_linear}{\em 
For an endomorphism $\Phi: E \to E$, 
\begin{equation*}
W_\Phi(v) = \left. \frac{d}{d\epsilon}\right|_{\epsilon=0} (v + \epsilon \, \Phi(v))
\end{equation*}
defines a linear vector field on $E$. The corresponding derivation $\Delta_{\Phi}: \Gamma(E^*) \to \Gamma(E^*)$ is just $\Phi^*$ and its dual $\Delta_{\Phi}^\top$ is $-\Phi$. Note that $\ker(\sharp) = \{\Delta_{\Phi}\,\,|\,\, \Phi:E \to E \text{ endomorphism}\}$
}
\end{remark}

A derivation $\Delta: \Gamma(E) \to \Gamma(E)$ defines an operator $L_\Delta: \Omega^k(M, E) \to \Omega^k(M,E)$ by
\begin{align}\label{der_Lie}
\begin{split}
(L_\Delta \vartheta)(X_1, \dots, X_k) & = \Delta(\vartheta(X_1, \dots, X_k))\\ 
& \hspace{-20pt} + \sum_{i=1}^k(-1)^{i} \vartheta([\sharp(\Delta), X_i], X_1, \dots, X_k).
\end{split}
\end{align}

\begin{proposition}\label{prop:der_eq}
For $\vartheta \in \Omega^k(M, E)$ and a derivation $\Delta: \Gamma(E^*) \to \Gamma(E^*)$, one has that
$$
(L_\Delta \vartheta)(X_1, \dots, X_k) = \left.\frac{d}{d\epsilon}\right|_{\epsilon=0} \Fl_{W_{\Delta}}^{-\epsilon}(\vartheta(T\Fl_{\sharp(\Delta)}(X_1), \dots, T\Fl_{\sharp(\Delta)}(X_k))).
$$
\end{proposition}

\begin{proof}
 Given $\psi \in \Gamma(E^*)$,
\begin{align*}
\left.\frac{d}{d\epsilon}\right|_{\epsilon=0} \!\!\!\!\!\ell_\psi(\Fl_{W_{\Delta}}^{-\epsilon}(\vartheta(T\Fl^\epsilon_{\sharp(\Delta)}(X_1), \dots, T\Fl^\epsilon_{\sharp(\Delta)}(X_k)))) & =  \left.\frac{d}{d\epsilon}\right|_{\epsilon=0} \!\!\!\!\!\ell_\psi(\Fl_{W_{\Delta}}^{-\epsilon}(\vartheta(X_1, \dots, X_k)))\\
& \hspace{-180pt} + \left.\frac{d}{d\epsilon}\right|_{\epsilon=0} \!\!\!\!\!\ell_\psi(\vartheta(T\Fl^\epsilon_{\sharp(\Delta)}(X_1), \dots, T\Fl^\epsilon_{\sharp(\Delta)}(X_k)))\\
& \hspace{-120pt} = - \<\Delta(\psi), \vartheta(X_1, \dots, X_k)\> + (\Lie_{\sharp(\Delta)} \vartheta^*\psi)(X_1,\dots, X_k), 
\end{align*} 
where $\vartheta^*\psi \in \Omega^k(M)$ and we have used the properties of the tangent lift recalled in Example \ref{ex:tan_lift}. The result now follows from the definition \eqref{dual_der} of $\Delta^\top$.
\end{proof}

Our last result explores the correspondence between linear vector fields and derivations to give criteria for a linear vector field to be tangent to a subbundle.

\begin{proposition}\label{prop:derivation}
Let $E_0 \subset E$ be a vector subbundle over $M$ and $u \in \Gamma(E_0)$. A linear vector field $W \in \mathfrak{X}(E)$  satisfies $W(u(x)) \in T_{u(x)}E_0$ if and only if the adjoint of the corresponding derivation, $\Delta^\top: \Gamma(E) \to \Gamma(E)$, satisfies $\Delta^\top(u)(x) \in E_0$.
\end{proposition}

\begin{proof}
Locally $E$ is given by coordinates $(x_i, \xi^k)$ with $E_0$ determined by $\xi^k=0, \,\, k > r$, where $r = \mathrm{rank}(E_0)$. A linear vector field is given by
$$
W(x,\xi) = X_i(x) \frac{\partial}{\partial x_i} + \Phi_j^k(x) \xi^j \, \frac{\partial}{\partial \xi^k}.
$$
An arbitrary section $u \in \Gamma(E)$ is determined by $\xi^k = u^k(x)$ and $\Delta^\top(u)$ is given by
$$
\xi^k = X_i(x) \frac{\partial u^k}{\partial x_i}(x) - \Phi_j^k(x) u^j(x).
$$
Now, for $u \in \Gamma(E_0)$, one has that $u^k(x) \equiv 0$, for $k > r$. In this case $W(u(x)) \in T_{u(x)}E_0$ if and only if $\Phi_j^k (x) u^j(x)= 0, \, \forall \, k > r$. This is clearly equivalent to $\Delta^\top(u)(x) \in E_0$.
\end{proof}

\end{document}